\documentclass[11pt,letterpaper, reqno]{amsart}
\usepackage{graphicx} 
\usepackage{xcolor}
\usepackage{amsmath, amssymb, amsthm, amsrefs}
\usepackage{soul}
\usepackage{tikz}
\usepackage[colorlinks=true, pdfstartview=FitV, linkcolor=magenta, citecolor=blue, urlcolor=blue]{hyperref}
\usepackage[capitalise]{cleveref}
\usepackage{booktabs}
\usepackage[normalem]{ulem}
\usepackage[foot]{amsaddr}

\newtheorem{thm}{Theorem}[section]   
\newtheorem{cor}[thm]{Corollary}     
\newtheorem{lem}[thm]{Lemma}         
\newtheorem{prop}[thm]{Proposition}  

\theoremstyle{definition} 
\newtheorem{defn}[thm]{Definition}   
\newtheorem{conj}[thm]{Conjecture}        
\newtheorem{ex}[thm]{Example}        
\newtheorem{rmk}[thm]{Remark}
\newtheorem{question}[thm]{Question}

\Crefname{lem}{Lemma}{Lemmas}
\Crefname{defn}{Definition}{Definitions}
\Crefname{thm}{Theorem}{Theorems}
\Crefname{prop}{Proposition}{Propositions}

\DeclareMathOperator{\im}{im}

\DeclareMathOperator{\tk}{Tk}

\DeclareMathOperator{\gens}{gens}

\DeclareMathOperator{\sep}{sep}
\DeclareMathOperator{\brank}{brank}
\DeclareMathOperator{\mrank}{mrank}
\DeclareMathOperator{\supp}{supp}

\def\B{\mathbb{B}}
\def\F{\mathbb{F}}
\def\N{\mathbb{N}}

\def\bc{\mathbf{c}}
\def\bd{\mathbf{d}}

\def\bx{\mathbf{x}}
\def\by{\mathbf{y}}

\def\bA{\mathbf{A}}
\def\bB{\mathbf{B}}
\def\bC{\mathbf{C}}
\def\bD{\mathbf{D}}
\def\bH{\mathbf{H}}
\def\bI{\mathbf{I}}

\def\bP{\mathbf{P}}
\def\bT{\mathbf{T}}

\def\bV{\mathbf{V}}
\def\bX{\mathbf{X}}

\def\cC{\mathcal{C}}
\def\cD{\mathcal{D}}
\def\cE{\mathcal{E}}

\def\cV{\mathcal{V}}

\def\pcode{\mathbf{P_{Code}}}
\def\code{\mathbf{Code}}

\title{Adjoints of Morphisms of Neural Codes}
\author{Juliann Geraci$^1$}
\author{Alexander B. Kunin$^{2}$}
\author{Alexandra Seceleanu$^1$}
\address[1]{Department of Mathematics, University of Nebraska -- Lincoln}
\address[2]{Department of Mathematics, Creighton University}
\email{jgeraci2@huskers.unl.edu}
\email{alexkunin@creighton.edu}
\email{aseceleanu@unl.edu}
\date{March 10, 2026}

\begin{document}

\begin{abstract} 
A combinatorial code $\cC$ is a collection of subsets of $[n]$, or equivalently a set of points in $\{0,1\}^n$.
    A morphism of codes is a map from one combinatorial code to another such that the coordinates of points in the image can be expressed as products of coordinates in the domain.
    By representing morphisms of codes as binary matrices, we show that any morphism of codes is part of a Galois connection where its adjoint is boolean multiplication by the representative matrix. We use this to characterize those morphisms of codes which allow to factor a boolean matrix, with applications to estimating boolean matrix rank. 
    
    Morphisms also induce a partial order on (isomorphism classes of) codes.
    We determine the covering relations in this partial order for which the two adjoint maps are mutual inverses in terms of \emph{free} neurons, a combinatorial condition on the index corresponding to the covering maps.
    We introduce the \emph{defect} of a code as a new tool to study this poset and show that defect decreases by exactly 0 or 1 under a covering map.
\end{abstract}
\maketitle

\setcounter{tocdepth}{1}
\tableofcontents

\section{Introduction}

The neural ring was introduced by Curto, Itskov, Veliz-Cuba, and Youngs to study the intrinsic structure of combinatorial neural codes~\cite{neuralring13}.
A combinatorial neural code is an affine variety the field of order 2, which we denote $\F$, and the neural ring is the coordinate ring of this variety.
The ``intrinsic structure'' is encoded by a specific set of minimal generators of the vanishing ideal of the variety. The goal of the authors was to use this intrinsic structure to identify which combinatorial codes arise as intersection patterns of convex sets in some Euclidean space (see, e.g.~\cites{curto2017whatconvex, curto2019algebraic}).

Towards this goal, Jeffs introduced a class of maps between codes, dubbed morphisms of neural codes, such that the image of a convex code is convex~\cite{morphisms}. This induces a graded partial order on isomorphism classes of codes, denoted $\pcode$, such that convex codes form a down-set.
In this article, we present some advancements in the understanding of this partial order and apply it to the problem of Boolean matrix factorization.

Boolean matrix factorization takes as input a given $m\times n$ binary matrix $\bC$ and asks for a pair of matrices $\bV, \bH$ such that $\bC = \bV\bH$, where addition and multiplication are carried out over the Boolean semiring $\B$ (that is, $\B = \{0,1\}$ equipped with the operations $\vee$ and $\wedge$ of Boolean logic). We observed that, in some but not all cases, reading the map from rows of $\bC$ to rows of $\bV$ yielded a morphism of codes.
Indeed, the form of $\bV$ in those cases had been described as early as the 1950s \cites{luce1952note,plemmons1971generalized} in the context of solving the equation $\bC = \bV\bH$ for $\bV$.
Conversely, writing the domain and image of a morphism of codes in matrix form sometimes, but not always, results in a Boolean matrix factoriation.

We set out to characterize those factorizations which correspond to morphisms of codes. To this end, we find two complementary interpretations of a given $r \times n$ matrix $\bH$: on the one hand, boolean multiplication by $\bH$ on the right yields a map $\B^r \to \B^n$. On the other hand, interpreting $\bH$ as a monomial map between neural rings, it is the pullback of a morphism of codes $\F^n \to \F^r$.
These maps are adjoint in the sense that they form a Galois connection, and this forms the basis for the present work.

We present four main results.
First, we show that the canonical form of the intersection completion of a code is a subset of the canonical form of the original code (\cref{thm:intersectioncompletionviaCF}).
A code is intersection-complete if for any pair of codewords their coordinate-wise product is also a codeword.
Intersection-complete codes have held a distinguished place in the study of combinatorial codes: they were among the first class of codes to be known to be convex~\cite{cruz2019open}, every intersection-complete code is image of a simplicial complex under a morphism of codes~\cite{morphisms}, and the vanishing ideal of an intersection-complete code is generated by monomials and binomials~\cite{curto2019algebraic}. Building on recent results of Geller and R.G.~\cite{Geller2024}, we show that the subset of monomial and binomial generators in the ideal of an arbitrary code generates the vanishing ideal of the intersection completion of that code.

Second and third, we show that the pair of maps defined by a matrix $\bH$ form a Galois connection (\cref{lem:galoisconnection}) and characterize for which covering relations in $\pcode$ this connection is a bijection (\cref{thm:coveringBMF}).
These are morphisms for which the inverse can be expressed by matrix multiplication, partially resolving our goal. The characterization we give is in terms of \emph{free} neurons, a combinatorial property which can be checked in time linear in the size of a matrix representing a code (\cref{lem:redundancylemma}).

Fourth and finally, we introduce the \emph{defect} of a code, an isomorphism invariant which measures deviation from intersection completeness and induces a weak bigrading on $\pcode$ (\cref{thm:bigrading}).
Further, bijective morphisms of codes (other than isomorphisms) strictly decrease defect.

The organization of the paper is as follows:
In \cref{sec:preliminaries} we establish the terminology, notation, and basic results we will use regarding combinatorial neural codes, with several technical lemmas regarding trunks of combinatorial codes deferred to \cref{sec:trunklemmas}.
\Cref{sec:canonicalforms} is a brief digression to prove \cref{thm:intersectioncompletionviaCF} which is of independent significance.
\Cref{sec:morphismsasmatrices} establishes the Galois connection between morphisms of codes and matrix multiplication. The properties of this connection are expanded on in \cref{sec:adjoints};
in particular in \cref{sec:adjointsofcovering} we show which covering maps in $\pcode$ correspond to matrix factorizations.
In \cref{sec:defect} we define defect and show it is weakly order-preserving on $\pcode$.
\Cref{sec:booleanrank} applies our results to the problem of computing Boolean rank. \Cref{sec:conclusion} concludes with a brief discussion of our results and several open questions.

\section{Preliminaries}\label{sec:preliminaries}
    In this section we establish our terminology and notation, and review the relevant results on combinatorial codes.

    By a \emph{combinatorial code} we mean a collection $\cC \subseteq 2^{[n]}$ of subsets of $[n] := \{1,\dots,n\}$. We refer to elements of a code as \emph{codewords} and the elements of the ground set $[n]$ as \emph{neurons}, as our work builds on that of Curto et al.~\cite{neuralring13} and Jeffs~\cites{morphisms,jeffs2019sunflowers}.
    We can also view a code as a set of binary vectors $\cC \subseteq \{0,1\}^n$ by encoding a vector via its \emph{support}: for $\bc = (c_1,\dots,c_n) \in \{0,1\}^n$, $\supp(\bc) := \{i \in [n] \mid c_i = 1\}$.
    Note that $2^{[n]}$ comes equipped with the partial order $\subseteq$ of set containment. 
    We can order $\{0,1\}^n$ by setting $0 \leq 1$ and then setting $\bc \leq \bd$ if $c_i \leq d_i$ for all $i$. Under this order, $\bc \leq \bd$ if and only if $\supp(\bc) \subseteq \supp(\bd)$.

    We translate between codes and matrices by writing each codeword as a row of the matrix. Thus, a code $\cC = \{\sigma_1,\dots,\sigma_m\}$ consisting of $m$ subsets of $[n]$ specifies an $m \times n$ matrix
    \[ \bC = [c_{ij}] \quad\text{with}\quad c_{ij} = \begin{cases}
        1 & j \in \sigma_i\\
        0 & j \notin \sigma_i.
    \end{cases}\]
    Since a given matrix may have repeated rows, however, an $m \times n$ matrix $\bC$ might specify a code with fewer than $m$ codewords.

    We adopt the following notational conventions: When discussing a code as a set system, it is denoted with calligraphic capital letters, e.g.\ $\cC = \{\sigma_1,\dots,\sigma_m\} \subseteq 2^{[n]}$. As here, lowercase Greek letters will denote subsets of $[n]$. When clear from context, we will drop the set braces when discussing sets of one or two elements, i.e.\ we write $i$ for $\{i\}$ or $ij$ for $\{i,j\}$ in order to reduce notational clutter.
        
    Binary matrices and vectors are denoted with bold letters, e.g.\ $\bC \in \{0,1\}^{m\times n}$.
    Given a vector $\bc = (c_1,\dots,c_n)$ and an index set $\sigma \subseteq [n]$, we denote the product of elements indexed by $\sigma$ as $\bc^\sigma := \prod_{i\in\sigma} c_i$. For a set of variables in a polynomial ring we denote monomials the same way, $x^\sigma = \prod_{i\in\sigma} x_i$, and in addition we use the notation $(1 - x)^\sigma := \prod_{i\in\sigma} (1- x_i)$.

    In what follows, much of what we write is most succinctly presented by viewing $\cC \subseteq2^{[n]}$ as a collection of finite sets; however, when it is not too cumbersome, we will also state results and formulas in terms of binary vectors. The remainder of this section is devoted to review of the relevant definitions and theorems from \cites{neuralring13, morphisms, jeffs2019sunflowers}.

    \subsection{The coordinate ring and canonical form of a code}
    The set $\{0,1\}$ may be endowed with algebraic structure in which $0$ is the additive identity in two ways. We denote by $\B$ the semiring $\{0,1\}$ with the operations of Boolean logic (that is, $1 \vee 1 = 1$),
    and by $\F$ the field $\{0,1\}$ with arithmetic mod 2 (that is, $1 + 1 = 0$). Thus we may also view codes as subsets of $\B^n$ or $\F^n$ if we wish to endow them with additional structure.
        
        In particular, 
        it can be fruitful to think of $\cC \subseteq \F^n$ as an affine variety. Letting $R = \F[x_1,\dots,x_n]$ denote the polynomial ring in $n$ variables, the vanishing ideal of $\cC$ is denoted $I_\cC$,
            \[ I_\cC = \{ f \in R : f(\bc) = 0 \text{ for all } \bc \in \cC\}. \]
        The ideal $I_\cC$ is generated by \emph{pseudomonomials}, polynomials of the form $x^\sigma(1-x)^\tau$. We say a pseudomonomial is \emph{square-free} if $\sigma \cap \tau = \varnothing$. The \emph{canonical form} of $I_\cC$, denoted $CF(I_\cC)$, is the set of minimal (by divisibility) squarefree pseudomonomials in $I_\cC$.
        We sometimes abuse terminology and refer to $CF(I_\cC)$ as the \emph{canonical form of $\cC$}.

        We will call the quotient ring $R_\cC := R / I_\cC$ the \emph{neural ring of $\cC$} following \cite{neuralring13}.
        The neural ring is the ring of $\F$-valued functions on $\cC$.
        This ring comes equipped with a distinguished set of functions, the coordinate functions $x_1,\dots,x_n$, defined by
            \[ x_i(\bc) = c_i \quad \text{or, in terms of sets,}\quad x_i(\sigma) = \begin{cases}
                1 & i\in\sigma\\
                0 & i\notin\sigma.
            \end{cases} \]

        \begin{ex}\label{ex:standard1}
            Let $\cC = \{\varnothing, 12,23,34,123,234,1234\}$ (to avoid notational clutter, we write subsets of $[n]$ without set brackets or commas; that is, $12 = \{1,2\}$). In matrix form,
                \[ \bC = \begin{bmatrix}
                    0 & 0 & 0 & 0\\
                    1 & 1 & 0 & 0\\
                    0 & 1 & 1 & 0\\
                    0 & 0 & 1 & 1\\
                    1 & 1 & 1 & 0\\
                    0 & 1 & 1 & 1\\
                    1 & 1 & 1 & 1\\
                \end{bmatrix}
                .
                \]
            The vanishing ideal $I_\cC$ is generated by
            \begin{align*}
                I_\cC = \langle & x_1(1-x_1), x_2(1-x_2), x_3(1-x_3), x_4(1-x_4),\\
                    & x_4 (1 - x_3), x_1 (1 - x_2), \\
                    & x_2 (1 - x_1) (1 - x_3), x_3 (1 - x_2) (1 - x_4) \rangle
            \end{align*}
            The canonical form of $I_\cC$ is comprised of the last two rows of generators.
        \end{ex}
    
    \subsection{Morphisms of codes}
        For a code $\cC \subseteq 2^{[n]}$ and any subset $\tau \subseteq \N$, the \emph{trunk} of $\tau$ in $\cC$ is the subset of $\cC$ specified by
            \[ \tk_\cC(\tau) = \{\sigma \in \cC : \sigma \supseteq \tau\} \quad \text{or, in binary,} \quad \tk_\cC(\tau) = \{\bc \in \cC : \bc^\tau = 1\}. \]
        Note, we do not require $\tau \in \cC$, and allow, e.g.\ $\tk_\cC([n+1]) = \varnothing$.
        A trunk of a code is any subset of the above form. In particular, the empty set is a trunk of any code, and a code is a trunk of itself as $\cC = \tk_\cC(\varnothing)$.
        If $T = \tk_\cC(i)$ for some $i \in [n]$, we say $T$ is a \emph{simple trunk}.

        A \emph{morphism of codes} is a function $f: \cC \subseteq 2^{[n]} \to \cD \subseteq 2^{[r]}$ such that the preimage of a trunk (in $\cD$) is a trunk (in $\cC$).
        Any morphism of codes is completely specified by the preimages of simple trunks; in other words, a morphism can be specified by listing an ordered tuple $(T_1,\dots,T_r)$ of trunks of $\cC$, where $T_i = f^{-1}(\tk_\cD(i))$ for each $i \in [r]$. The morphism is then defined by 
        \begin{align}
            f(\sigma) = \{ j \in [r] : \sigma \in T_j\}\label{eq:morphismtrunks}
        \end{align}
        (see \cite{morphisms}*{Definition 2.10 and Propositions 2.11 and 2.12}).
        Two codes $\cC, \cD$ are isomorphic, written $\cC \cong \cD$, if there is a bijective morphism of codes $\cC \to \cD$ whose inverse is also a morphism.
        
        If $\cC \to \cD$ is a surjective morphism of codes, we say $\cD \leq \cC$.
        This defines a preorder on all combinatorial codes, which induces a partial order on isomorphism classes of codes.
        We denote this poset $\pcode$.
        
        \begin{rmk}\label{rmk:pcode}
            This definition of $\pcode$ is presented in \cite{jeffs2021phd}, and differs slightly from that defined in \cite{morphisms}.
            In the latter, we also have $T \leq \cC$ for any trunk $T$ of $\cC$.
            As noted in \cite{jeffs2025coveringrelationsposetcombinatorial}*{Proposition 12}, we always have $T \cup \{\varnothing\} \leq \cC$: the map that sends elements of $T$ to themselves and elements outside of $T$ to $\varnothing$ is a morphism of codes.
            Therefore, qualitatively at least, these differing notions of $\pcode$ capture the same relationships among codes.
        \end{rmk}
    
    \subsection{Triviality and redundancy}
        We will need the following technical notions.
                
        \begin{defn}
            Let $\cC \subseteq 2^{[n]}$ be a code, and $i \in [n]$ a neuron. We say $i$ is \emph{trivial} if $\tk_\cC(i) = \varnothing$. If there is a set $\sigma \subseteq[n]\setminus i$ such that $\tk_\cC(i) = \tk_\cC(\sigma)$, we say \emph{$i$ is redundant to $\sigma$}. We simply say $i$ is \emph{redundant} if there exists such a $\sigma$.
            If $\cC$ has no trivial or redundant neurons, we say $\cC$ is \emph{reduced}.
        \end{defn}
        The following two lemmas follow immediately from this definition.
        \begin{lem}
            Let $\cC$ be a code. The following are equivalent:
            \begin{enumerate}
                \item Neuron $i$ is trivial.
                \item In a matrix representing $\cC$, the $i$-th column is all 0s.
                \item $x_i \in I_\cC$, that is, $x_i = 0$ in $R_\cC$.
            \end{enumerate}
        \end{lem}
        \begin{lem}\label{lem:characterizeredundant}
            Let $\cC$ be a code. The following are equivalent:
            \begin{enumerate}
                \item Neuron $i$ is redundant to $\sigma$.
                \item Column $i$ is the entry-wise product of columns indexed by $\sigma$.
                \item $x_i - x^\sigma \in I_\cC$, that is, $x_i = x^\sigma$ in $R_\cC$.
            \end{enumerate}
        \end{lem}

        \begin{rmk}\label{rmk:reductionmorphism}
            If two codes are reduced and isomorphic, the isomorphism is a permutation of neurons \cite{morphisms}*{Corollary 3.8}. 
            If a code is not reduced, then up to permutation there is a unique reduced representative \cite{morphisms}*{Theorem 1.4}, and it can be found by first deleting all trivial neurons, then deleting redundant neurons one at a time until no redundant neurons are left.
            In light of this, we refer to the composition of this sequence of deletions as \emph{the} reduction morphism and denote it $\rho$.
        \end{rmk}
        From this remark, we get the following.
        \begin{lem}\label{lem:rhoisprojection}
            Let $\cC$ be a code and $\bC$ a matrix representing it. Then there is a projection matrix $\bP$ such that $\rho(\cC)$ is represented by $\bC\bP$.
        \end{lem}
        \noindent The intended reading of this lemma is in terms of Boolean matrix operations, but we note that this statement holds over $\F$ as well.

\section{Canonical forms of neural ideals}\label{sec:canonicalforms}

    We now take a brief digression to state a pair of theorems that may be of independent interest to those studying the convex code problem \cites{curto2017whatconvex,curto2019algebraic}.
    Intersection-complete and union-complete codes will play an important role throughout our work.

\subsection{Intersection- and Union-complete codes}
    A code $\cC \subseteq 2^{[n]}$ is called \emph{intersection-complete} if for any pair of codewords $\sigma, \tau \in \cC$, their intersection $\sigma\cap\tau$ is a codeword, $\sigma\cap\tau \in \cC$.
Likewise, a code is called \emph{union-complete} if $\sigma\cup\tau \in\cC$ for every pair of codewords $\sigma,\tau \in \cC$.

    The authors of~\cite{curto2019algebraic} proved that the canonical form of the vanishing ideal of an intersection-complete code is at most linear in the $(1-x_i)$s. We state their theorem below together with a dual statement about union-complete codes, and provide a slightly different proof.

    \begin{prop}[\cite{curto2019algebraic}*{Proposition 3.7}]\label{lem:intgens}
        A code $\cC$ is intersection-closed if and only if the canonical form of $\cC$ is at most linear in the $(1-x_i)$ terms. That is, if $x^\sigma(1-x)^\tau \in CF(I_\cC)$, then $|\tau| \leq 1$.
    \end{prop}
    \begin{prop}\label{lem:uniongens}
        A code $\cC$ is union-closed if and only if the canonical form of $\cC$ is at most linear in the $x_i$ terms. That is, if $x^\sigma(1-x)^\tau \in CF(I_\cC)$ then $|\sigma| \leq 1$.
    \end{prop}
    \begin{proof}[Proof of \cref{lem:intgens,lem:uniongens}]
		(\ref{lem:intgens}) We will show that any pseudomonomial $x^\sigma(1-x)^\tau \in I_\cC$ with $|\tau| > 1$ cannot be in the canonical form.
        
		Suppose $p(x) = x^\sigma (1-x)^\tau \in I_\cC$, so $p(c) = 0$ for all $c \in \cC$.
        If there are no codewords of $\cC$ that contain $\sigma$, then we would have $x^\sigma \in I_\cC$, and so $p(x)$ cannot be minimal unless $\tau = \varnothing$.
        So, suppose $\sigma$ is contained in some codeword of $\cC$, and let $s$ be the intersection of all codewords that contain $\sigma$.
        Since $\cC$ is intersection-closed, $s\in\cC$ and so $p(s) = 0$.
        Since $x_i = 1$ for $i \in s$, we must have $s \cap \tau \neq \varnothing$.
        In particular, this shows that any codeword that contains $\sigma$ must contain $\sigma \cup j$ for each $j \in s \cap \tau$, and therefore $x^\sigma(1-x_j) \in I_\cC$ for each $j \in s\cap \tau$.
        This shows that $x^\sigma(1-x)^\tau$ is not minimal and hence not in the canonical form if $|\tau| > 1$.
        
		(\ref{lem:uniongens}) We will show that any pseudomonomial $x^\sigma(1-x)^\tau \in I_\cC$ with $|\sigma| > 1$ cannot be in the canonical form.
		
		Suppose $q(x) = x^\sigma (1-x)^\tau \in I_\cC$, so $q(c) = 0$ for all $c \in \cC$.
        If $\tau$ is not disjoint from any codewords of $\sigma$, then $(1-x)^\tau \in I_\cC$ and so $q(x)$ cannot be minimal unless $\sigma = \varnothing$.
        So, suppose $\tau$ is disjoint from some codeword of $\cC$ and let $t$ be the union of all codewords disjoint from $\tau$.
        Since $\cC$ is union-closed, $t \in \cC$ and so $q(t) = 0$. But since $t$ is disjoint from $\tau$, we need one of the $x_i$ factors to evaluate to 0 at $t$; in other words, $\sigma \setminus \tau \neq \varnothing$.
        Therefore, any codeword which is disjoint from $\tau$ must not contain $j$ for each $j\in \sigma \setminus \tau$, and so $x_j(1-x)^\tau \in I_\cC$ for each $j \in \sigma\setminus \tau$.
        But in particular, this shows that $x^\sigma (1-x)^\tau$ is not minimal when $|\sigma| > 1$.
	\end{proof}

\subsection{Computing completions}
    \Cref{lem:intgens,lem:uniongens} can be extended: Given a code $\cC$, discarding the generators of $I_\cC$ which do not meet the conditions of \cref{lem:intgens} (respectively, \cref{lem:uniongens}) yields the canonical form of the intersection completion (resp., union completion) of $\cC$.

    To prove this, we will make use of Geller and R.~G.'s characterization of canonical forms of neural ideals\cite{Geller2024}.
    For $p,q$ two squarefree pseudomonomials in $\F[x_1,\dots,x_n]$, define the \emph{separator} of $p,q$ to be $\sep(p,q) = \{i : x_i(1-x_i) \mid pq\} \subseteq [n]$.

    \begin{thm}[Restatement of \cite{Geller2024}*{Theorem 1.1}]\label{thm:whencanonical}
		Let $\Gamma = \{p_1,\dots,p_m\}$ be a set of squarefree pseudomonomials such that $p_i \nmid p_j$ for all $i \neq j$. Then $\Gamma$ is in canonical form if and only if, for any pair $p \neq q \in \Gamma$, if $\sep(p,q) = \{i\}$, there is some $r \in \Gamma$ such that $r \mid\left( pq / x_i(1-x_i) \right)$.
	\end{thm}

   This leads to our first main results determining the canonical forms of intersection and union completions of codes.

    \begin{thm}\label{thm:intersectioncompletionviaCF}
        Let $\cC \subseteq 2^{[n]}$ be a code and $\Gamma = \{p_1,\dots,p_t\}$ the canonical form of $I_\cC$. Set $\hat\Gamma = \{x^\sigma (1-x)^\tau \in \Gamma : |\tau| \leq 1\}$. Then $\hat\Gamma$ is the canonical form of the vanishing ideal of the intersection completion of $\cC$.
    \end{thm}
    \begin{proof}
        First, we will verify the two conditions of \cref{thm:whencanonical}.
		Since we started with $\Gamma$ in canonical form and removed some elements, none of the $p_i$ divide each other, so it only remains to check the condition on separators.
		
		Suppose $p, q \in \hat{\Gamma}$ with $\sep(p,q) = \{i\}$.
		Without loss of generality, $p = x^\alpha (1-x_i)$ for some $\alpha \subseteq [n] \setminus \{i\}$ and $q = x^{\sigma \cup i}(1-x)^\tau$ with $\sigma, \tau \subseteq [n] \setminus i$ and $|\tau| \leq 1$.
		By \cref{thm:whencanonical}, there must be some $r \in \Gamma$ such that 
        \[
        r \mid \frac{pq}{x_i(1-x_i)} = x^{\alpha \cup \sigma}(1-x)^\tau.
        \]
        Such an $r$ would also be in $\hat{\Gamma}$. Therefore, by \cref{thm:whencanonical}, $\hat{\Gamma}$ is in canonical form.

        Now, let $\hat\cC \subseteq \F^n$ denote the zero set of the ideal generated by $\hat\Gamma$, and $\hat{I}$ the vanishing ideal of $\hat\cC$. By the ideal-variety correspondence, since $\hat{I} \subseteq I_\cC$ we have $\hat\cC \supseteq \cC$, and by  \cref{lem:intgens}, $\hat\cC$ is intersection-complete and thus contains the intersection-completion of $\cC$.

        Now consider $\sigma \in \hat\cC \setminus \cC$. First, we claim that $\sigma$ must be a contained in at least one codeword of $\cC$: if not, then $x^\sigma \in I_\cC$, so some monomial in the canonical form of $I_\cC$ divides $x^\sigma$. This monomial would be in $\hat\Gamma$ and therefore in $\hat I$, so $\sigma \notin \hat\cC$. Thus $\sigma$ is contained in at least one codeword of $\cC$.
        
        Let $\Sigma$ be the intersection of all codewords in $\cC$ that contain $\sigma$. If $\sigma \subsetneq \Sigma$, then the argument used to prove  \cref{lem:intgens} shows that for each $i \in \Sigma \setminus \sigma$, we must have $x^\sigma(1-x_i) \in I_\cC$. However, each of these terms is also in $\hat{I}$ and therefore $\sigma \notin \hat\cC$. Therefore, we must have $\sigma = \Sigma$ and so every word in $\hat\cC$ is the intersection of words in $\cC$; that is, $\hat\cC$ is the intersection completion of $\cC$.
    \end{proof}

    This argument can be applied, \emph{mutatis mutandis}, to prove the same result applies to the union-completion:
    \begin{thm}
        Let $\cC \subseteq 2^{[n]}$ be a code, and $\Gamma = \{p_1,\dots,p_t\}$ the canonical form of $I_\cC$. Set $\Gamma' = \{x^\sigma (1-x)^\tau \in \Gamma : |\sigma| \leq 1\}$. Then $\Gamma'$ is the canonical form of the vanishing ideal of the union completion of $\cC$.
    \end{thm}

\section{Morphisms as matrices}\label{sec:morphismsasmatrices}
    The observation at the heart of our work is that morphisms of codes and monomial maps can be specified by binary matrices, and composition of monomial maps is given by the \emph{Boolean} product of those matrices.
    
\subsection{Representatives of morphisms}
    If $T \subseteq \cC$ is a trunk, then there may be more than one $\tau \subseteq [n]$ such that $T = \tk_\cC(\tau)$.
    If $T = \tk_\cC(\tau)$, we say $\tau$ \emph{generates} the trunk $T$.
    The \emph{generators} of a trunk is the set
        \[ \gens(T) = \{\tau \subseteq [n] : T = \tk_\cC(\tau)\}. \]
    We adopt the convention that if $T$ is empty then $\gens(T) = \{\}$ is the empty set.
    
    \begin{lem}\label{lem:root}
        For any nonempty trunk $T \subseteq \cC$, $\gens(T)$ has a unique maximal element given by $\tau = \bigcap_{\alpha \in T} \alpha$.
    \end{lem}
    
    \begin{proof}
        It is clear that $T \subseteq \tk_\cC(\tau)$ as every word in $T$ contains $\tau$.
        If $\tau' \in \gens(T)$ then it is a subset of every codeword in $T$, and therefore it is a subset of their intersection, i.e.\ $\tau' \subseteq \tau$.
        Thus, we must have $\tau \in \gens(T)$ and $\tau$ is maximal in $\gens(T)$.
    \end{proof}
    
    In light of the above, we introduce the following terminology and notation:
    
    \begin{defn}\label{def:root}
        Let $S \subseteq 2^{[n]}$ be nonempty. 
        Then the \emph{root of $S$} is the intersection of all elements of $S$, that is $\sqrt{S} = \bigcap_{\sigma \in S} \sigma \subseteq [n]$.
        Given a code $\cC\subseteq 2^{[n]}$ and $\sigma \subseteq [n]$, we define the \emph{root of $\sigma$ relative to $\cC$} as $\sqrt[\cC]{\sigma} = \sqrt{\tk_\cC(\sigma)}$.
        
        For a trunk $T \subseteq \cC$, we say $T$ is \emph{rooted} if $\sqrt{T} \in \cC$ and \emph{free} if $\sqrt{T} \notin \cC$. For $i \in [n]$, we say $i$ is rooted if $\sqrt[\cC]{i} \in \cC$ and free if $\sqrt[\cC]{i} \notin \cC$.
    \end{defn}
        
    \begin{defn}
    Let $f: \cC \subseteq 2^{[n]} \to \cD \subseteq 2^{[r]}$ be a morphism of codes such that for each $i \in [r]$, $f^{-1}(\tk_\cD(i))$ is nonempty.
    A \emph{representative of $f$} is an ordered tuple $(\tau_1,\dots,\tau_r)$ of subsets of $[n]$ such that, for each $i \in [r]$, $\tau_i \in \gens(f^{-1}(\tk_\cD(i)))$.
    Alternatively, the $r\times n$ matrix where the support of the $i$th row is $\tau_i$ represents $f$:
        \begin{equation}\label{eq: T}
             \bT = [t_{ij}], \quad t_{ij} = \begin{cases}
            1 & j \in \tau_i\\
            0 & j \notin \tau_i
        \end{cases} 
        \end{equation}
    \end{defn}
    The condition that $f^{-1}(\tk_\cD(i))$ be nonempty is necessary due to our convention that $\gens(T) = \varnothing$ when $T$ is empty.
    
    If $(\tau_1,\dots,\tau_r)$ is a representative of a morphism $f$, formula \eqref{eq:morphismtrunks} can be written
        \begin{align}
            f(\sigma) = \{j \in [r] : \tau_j \subseteq \sigma\}. \label{eq:morphismsets} 
        \end{align}
    For binary vectors $\bc = (c_1,\dots,c_n) \in \{0,1\}^n$, we can rewrite \eqref{eq:morphismsets} as
        \begin{align}
            f(\bc) = \Bigl(\prod_{i \in \tau_1} c_i, \dots, \prod_{i \in \tau_r} c_i \Bigr)=\Bigl( \bc^{\tau_1}, \dots, \bc^{\tau_r} \Bigr). \label{eq:morphismvecs}
        \end{align}
    Given a tuple $t = (\tau_1,\dots,\tau_r)$ of subsets of $[n]$ we denote by $f_t$ the map specified by \eqref{eq:morphismsets} or \eqref{eq:morphismvecs}, and call it the \emph{morphism represented by $t$.}

\subsection{The pullback homomorphism}\label{sec:pullback}
    Writing a morphism of codes $f$ in the form \eqref{eq:morphismvecs} makes it clear that the pullback homomorphism $f^*: R_\cD \to R_\cC$ acts on coordinate functions by $x_i \mapsto x^{\tau_i}$.
    Such a ring homomorphism, in which the image of each variable is a monomial, is called a \emph{monomial map}.
    Jeffs showed that the category whose objects are combinatorial codes and morphisms are morphisms of codes is equivalent to the category of neural rings with monomial maps~\cite{morphisms}.

    Thus, a matrix $\bT$ as in \eqref{eq: T} specifies a morphism of codes which is equivalent to a monomial map. Conversely, given a monomial map $\phi$, reading the exponent vector of $\phi(x_i)$ as the  $i$th row of a matrix gives a representative of $\phi$.
    We will denote this matrix $[\phi]$; that is, $[\phi]_{ij} = 1$ if $x_j \mid \phi(x_i)$ and $0$ otherwise.
    If $\phi: R_\cC \to R_\cD$ and $\psi: R_\cD \to R_\cE$ are monomial maps, a straightforward calculation shows that $[\phi \circ \psi] = [\psi][\phi]$. Because $x_i^2 = x_i$ for all $i$ in their respective neural rings, the matrix product is actually carried out over $\B$.

    In general, the matrix representing a monomial map need not be unique, as the following example shows.
    \begin{ex}\label{ex:nonuniquemonomialmap}
        Consider $\cC$ from \cref{ex:standard1}, and let $\cV$ be the code
        \begin{align*}
        \cV & = \{\varnothing, 1,2,3,12,23,123\}.
        \end{align*}
        The following matrices represent the same ring homomorphism $R_\cV \to R_\cC$:
        \[ \bH = \begin{bmatrix}
            1 & \mathbf{1} & 0 & 0\\
            0 & 1 & 1 & 0\\
            0 & 0 & 1 & 1
        \end{bmatrix}
        \qquad
        \bH' = \begin{bmatrix}
            1 & \mathbf{0} & 0 & 0\\
            0 & 1 & 1 & 0\\
            0 & 0 & 1 & 1
        \end{bmatrix}\]
        The difference is highlighted in bold in the first row. To see these specify the same map, it suffices to note that $x_1 = x_1x_2$ in $R_\cC$, as $x_1(1-x_2)$ is a generator of $I_\cC$.
    \end{ex}

    This example is simply a manifestation of the fact that both
    \[ (12,~23,~34)\quad \text{and} \quad(1,~23,~34)\] 
    are representative of the same morphism of codes $\cC \to \cV$.

\subsection{Galois connection}
    
    A Galois connection between two posets $A,B$ is a pair of functions $F: A \to B$, $G: B\to A$, such that
            \[ F(a) \leq b \quad \iff \quad a \leq G(b). \] 
        In this setting, $F$ is called the \emph{lower adjoint} and $G$ the \emph{upper adjoint}.
    
        The two maps defined by a matrix $\bH$ form a Galois connection (proven below in \cref{lem:galoisconnection}).
        In what follows, $H_a$ denotes the $a$th row of $\bH$ and $\eta_a = \supp(H_a)$. 
        \begin{defn}\label{def:galoisconnectionofH}
            Let $\bH$ be an $r \times n$ binary matrix. Then we define the following pair of maps between $2^{[r]}$ and $2^{[n]}$:
            \[ F_{\bH}: 2^{[r]} \to 2^{[n]} \quad \text{and} \quad G_{\bH}: 2^{[n]} \to 2^{[r]} \]
            \begin{itemize}
                \item $F_{\bH}$ is Boolean matrix multiplication:
                    \begin{align}
                        2^{[r]} & \to 2^{[n]} & \{0,1\}^r & \mapsto \{0,1\}^n\nonumber\\
                        \xi &\mapsto \bigcup_{j\in\xi} \eta_j & \bx &\mapsto \bx\bH\label{eq:galoisF}\\
                        & & (F_\bH(x))_i & = \bigvee_{j=1}^n x_j \wedge h_{ji}
                    \end{align}
                \item $G_{\bH}$ is the morphism of codes represented by $(\eta_1,\dots,\eta_r)$
                    \begin{align}
                        2^{[n]} & \to 2^{[r]} & \{0,1\}^n & \to \{0,1\}^r\nonumber\\
                        \xi & \mapsto \{j : \eta_j \subseteq \xi\}
                         & (x_1,\dots,x_n) & \to (x^{H_1},\dots, x^{H_r})\label{eq:galoisG}.
                    \end{align}
            \end{itemize}
        \end{defn}
    \noindent Note that $F_\bH$ is not, in general, a morphism of codes.
    \begin{lem}\label{lem:galoisconnection}
            The maps $F_{\bH},G_{\bH}$ defined above are the lower and upper adjoints, respectively, in a Galois connection between $2^{[n]}$ and $2^{[r]}$.
        \end{lem}
        
        \begin{proof}
            Let $\alpha \subseteq [r], \sigma \subseteq [n]$. Suppose $F_{\bH}(\alpha) \subseteq \sigma$. This means $\bigcup_{j\in\alpha} \eta_j \subseteq \sigma$. Therefore $\alpha \subseteq \{k : \eta_k \subseteq \sigma\} = G_{\bH}(\sigma)$

            For the converse, suppose $\alpha \subseteq G_\bH(\sigma)$. This means $\alpha \subseteq \{k : \eta_k \subseteq \sigma\}$. Then this gives us 
                \[ F_\bH(\alpha) = \bigcup_{j\in\alpha} \eta_j \subseteq \bigcup_{\{k: \eta_k \subseteq \sigma\}} \eta_k \subseteq \sigma.\]
            Thus we have shown $F_\bH(\alpha) \subseteq \sigma$ iff $\alpha \subseteq G_\bH(\sigma)$.
        \end{proof}

        The relationship between morphisms of codes and Boolean matrix multiplication suggests that there should be a way to express a morphism of codes in terms of Boolean algebra. Indeed, letting $\lnot$ denote the complement operator (e.g.\ $(\lnot\bH)_{ij} = 1 -\bH_{ij}$), we have the following formula.
        \begin{lem}\label{lem:morphismasmatrixoperations}
            Let $\bH$ be a Boolean matrix, and $G_\bH$ defined as above. Then for $\bx \in \B^n$, 
            $G_\bH(\bx) = \lnot((\lnot\bx)\bH^\top)$
        \end{lem}
        \begin{proof}
            Note that $x_j^{h_{ij}}$ is equivalent to $\lnot h_{ij} \vee x_j$, from which it follows that
            \begin{align*}
                G_\bH(\bx)_i = \prod_{j=1}^n x_j^{h_{ij}} & = \bigwedge_{j=1}^n (x_j \vee \lnot h_{ij})  = \lnot\left(\bigvee_{j=1}^n (\lnot x_j \wedge h_{ij})\right)\qedhere
            \end{align*}
        \end{proof}

        From this we conclude the image of $F_\bH$ is union-complete and the image of $G_\bH$ is intersection-complete.
        \begin{lem}\label{lem:images}
            Let $\bH$ be a binary matrix and $F_\bH$, $G_\bH$ defined as above.
            \begin{enumerate}
                \item\label{lem:unionH} The image of $F_\bH$ is the union completion of $\bH$.
                \item\label{lem:intersectionH} The image of $G_\bH$ is the intersection completion of $\lnot\bH^\top$.
            \end{enumerate}
        \end{lem}
        \begin{proof}
            \cref{lem:unionH} is immediate from \eqref{eq:galoisF}.

            \cref{lem:intersectionH} follows from \cref{lem:morphismasmatrixoperations}. Letting $\bx$ range over all of $\B^n$, we see that the image of $G_\bH$ is the complement of the union-completion of $\bH^\top$ which is the intersection completion of the complement of $\bH^\top$.
        \end{proof}

        \begin{rmk}\label{rmk:residual}
            Luce showed that the maximal $\bV$ that satisfies $\bV\bH \leq \bC$ (where the comparison $\leq$ is carried out entrywise) is $\lnot((\lnot\bC) \bH^\top)$ \cite{luce1952note}*{Theorem 5.2}.
            Thus, the algebra of Boolean matrices is \emph{(right) residuated}: there is always a maximal $\bX$ satisfying $\bX\bA \leq \bB$.
            This matrix is denoted $\bB : \bA$. Therefore, from here we will write $F_\bH(\bx) = \bx\bH$ and $G_\bH(\bx) = \bx:\bH$.
        \end{rmk}
        
\section{Morphisms as matrix factorizations}\label{sec:adjoints}
    We have established that morphisms of codes and Boolean matrix multiplication form a Galois connection.
    Using this connection, we will now give a partial characterization those morphisms whose adjoint is their inverse.
    Then we will show this induces a weak bigrading on the poset $\pcode$.
    In the next section we apply these results to the problem of Boolean matrix factorization.
    
    \subsection{Factorizations giving morphisms} 
        Given a fixed matrix $\bH \in \B^{r \times n}$, we now establish which matrix products with $\bH$ give rise to morphisms of codes. The following proposition follows from basic properties of Galois connections \cite{Bergman}*{Lemma 6.3.1}, but we give a direct proof here.
        \begin{prop}\label{prop:imagefixedpoints}
            For $\bH$ a binary matrix, $F_\bH(\bx) = F_\bH(G_\bH(F_\bH(\bx)))$, that is, $\bx\bH = ((\bx\bH):\bH)\bH$.
        \end{prop}
        \begin{proof}
            Starting with the trivial statement $F_\bH(\bx) \leq F_\bH(\bx)$ and applying \cref{lem:galoisconnection} we have $\bx \leq G_\bH(F_\bH(\bx))$. From the definition it is clear $F_\bH$ is order-preserving, so $F_\bH(\bx) \leq F_\bH(G_\bH(F_\bH(\bx)))$. On the other hand, starting from $G_\bH(\by) \leq G_\bH(\by)$ we get $F_\bH(G_\bH(\by)) \leq \by$. Applying $F_\bH \circ G_\bH$ to $F_\bH(\bx)$ we get $F_\bH(G_\bH(F_\bH(\bx))) \leq F_\bH(\bx)$ and so equality follows.
        \end{proof}
        \begin{cor}\label{lem:useHtogetfactorization}
            Let $\bH$ be a $r \times n$ binary matrix and $F_\bH$ and $G_\bH$ as above. Let $\bC$ be a matrix representing the image of $F_\bH$, and let $\bV$ be obtained by applying $G_\bH$ to each row of $\bC$. Then $\bC = \bV\bH$ as a Boolean product.
        \end{cor}

        Note the order of operations here: Given $\bH$, we compute the image of $F_\bH$. Writing this as a matrix and applying $G_\bH$ to each row results in a Boolean matrix factorization of the image of $F_\bH$.
        However, it is \emph{not} the case that, reading the mapping between rows in an arbitrary matrix produce $\bC = \bV'\bH$ will yield a morphism of codes, as the following example shows.
        \begin{ex}
            Continuing \cref{ex:nonuniquemonomialmap}, 
            the following are distinct factorizations of $\bC$:
            \[\underbrace{\begin{bmatrix}
                0 & 0 & 0 & 0\\
                1 & 1 & 0 & 0\\
                0 & 1 & 1 & 0\\
                0 & 0 & 1 & 1\\
                1 & 1 & 1 & 0\\
                0 & 1 & 1 & 1\\
                1 & 1 & 1 & 1
                \end{bmatrix}}_{\mathbf{C}}
                =
                \underbrace{\begin{bmatrix}
                0 & 0 & 0\\
                1 & 0 & 0\\
                0 & 1 & 0\\
                0 & 0 & 1\\
                1 & 1 & 0\\
                0 & 1 & 1\\
                1 & \mathbf{1} & 1
                \end{bmatrix}}_{\mathbf{V}}
                \underbrace{\begin{bmatrix}
                1 & 1 & 0 & 0\\
                0 & 1 & 1 & 0\\
                0 & 0 & 1 & 1
                \end{bmatrix}}_{\bH}
                =
                \underbrace{\begin{bmatrix}
                0 & 0 & 0\\
                1 & 0 & 0\\
                0 & 1 & 0\\
                0 & 0 & 1\\
                1 & 1 & 0\\
                0 & 1 & 1\\
                1 & \mathbf{0} & 1
                \end{bmatrix}}_{\bV'}
                \underbrace{\begin{bmatrix}
                1 & 1 & 0 & 0\\
                0 & 1 & 1 & 0\\
                0 & 0 & 1 & 1
                \end{bmatrix}}_{\bH}
                \]
            The difference is highlighted by the bold entry in the bottom row. The map $\bC \to \bV$ is a morphism of codes (in fact, the morphism described in \cref{ex:nonuniquemonomialmap}), while the map $\bC \to \bV'$ is not.
        \end{ex}

        \begin{cor}\label{lem:Hmaximalfactorizationmorphism}
            Suppose $\bC = \bV\bH$ as a Boolean product of matrices.
            If $\bV = (\bV\bH):\bH$, then the map taking rows of $\bC$ to rows of $\bV$ is a morphism of codes.
        \end{cor}
        \begin{proof}
            By assumption $\bC = \bV\bH$ and the map $\bC \to \bC : \bH = (\bV\bH):\bH$ is a morphism of codes. Indeed, when the equation $\bV = (\bV\bH):\bH$ is satisfied, this is the morphism $G_\bH$ in \eqref{eq:galoisG}.
        \end{proof}
        \noindent We call a matrix $\bV$ with this property \emph{$\bH$-maximal}. 
        Continuing \cref{rmk:residual}, we note that $\bH$-maximal matrices also appear in \cite{plemmons1971generalized}*{Lemma 1.1}, in which Plemmons shows that for a fixed $\bC,\bH$, the equation $\bC = \bV\bH$ has a solution if and only if $\bV = \bC:\bH$ is a solution.

    \subsection{Adjoints of morphisms} 
        To extend our results to arbitrary morphisms, 
        we need a way to extend the domain and codomain of a morphism of codes to the full Boolean lattice. For any morphism $g: \cC \to \cD$ with $\cC \subseteq 2^{[n]}$ and $\cD \subseteq 2^{[m]}$, we will say $\tilde{g}: 2^{[n]} \to 2^{[m]}$ is a \emph{lift} of $g$ if $\tilde{g}|_\cC = g$.

        \begin{lem}
            Let $g: \cC \subseteq 2^{[n]} \to \cD \subseteq 2^{[r]}$ be a morphism of codes such that the image of $g$ does not have any trivial neurons.
            Then the representatives of $g$ are in bijection with the lifts of $g$.
        \end{lem}
        \begin{proof}
            Given a representative $t = (\tau_1,\dots,\tau_r)$ of $g$, let $g_t: 2^{[n]} \to 2^{[r]}$ be defined by
                \[ g_t(\sigma) = \{j \in [r] : \tau_j \subseteq \sigma\}\]
            (c.f. formula \eqref{eq:morphismsets}).
            By definition, $g_t|_\cC = g$,
            thus every representative of $g$ yields a lift of $g$.
            
            Suppose $t' = (\tau_1', \dots,\tau_r')$ differs from $t$ in position $j$, so $\tau_j \neq \tau_j'$.
            Then, without loss of generality, $\tau_j \not\subseteq \tau_j'$ and so $g_t(\tau_j) \neq g_{t'}(\tau_j)$
            and so the mapping from representatives to lifts is one-to-one.

            Finally, let $\tilde{g}$ be a lift of $g$. If $\tilde{t} = (\tilde{\tau}_1, \dots, \tilde{\tau}_2)$ is a representative of $\tilde{g}$, then $\tilde{t}$ is a representative of $g$ by definition.
            Thus the mapping from representative to lifts is onto, completing the proof. 
        \end{proof}

        \begin{defn}\label{def:canonicallift}
            Let $f: \cC \subseteq 2^{[n]} \to \cD \subseteq 2^{[k]}$ be a morphism of codes defined by trunks $T_1,\dots, T_k \subseteq \cC$.
            Set $\tau_i = \sqrt{T_i}$ for each $i \in [k]$.
            We define the \emph{canonical representative} of $f$ to be $(\tau_1,\dots,\tau_k)$, and \emph{canonical lift} of $f$ to be the morphism $f_t: 2^{[n]} \to 2^{[m]}$.
        \end{defn}

        \begin{rmk}\label{rmk:hatlift}
            We note that the canonical lift is an extension of the map presented in \cite{jeffs2025coveringrelationsposetcombinatorial}*{Lemma 13}, which shows that for any morphism $f: \cC \to \cD$ there is a unique map $\hat{f}: \widehat{\cC} \to \widehat{\cD}$ such that $\hat{f}|_\cC = f$. The underlying construction is identical, we simply present it using representatives of morphisms rather than discussing the trunks directly. 
        \end{rmk}

        \begin{defn}\label{def:morphismadjoint}
            Let $f: \cC \subseteq 2^{[n]} \to \cD \subseteq 2^{[k]}$ be a surjective morphism of codes. The \emph{canonical adjoint} of $f$ is the map $f^\top: \cD \to 2^{[n]}$ defined by
            \begin{align}
                f^\top(d) = \bigcup_{j\in d} \sqrt{f^{-1}(\tk_\cD(j))}\label{eq:ftop}
            \end{align}
            Note that the sets $\sqrt{f^{-1}(\tk_\cD(j))}$ are precisely the sets appearing in the canonical representative of $f$.
            We say \emph{$f$ is a Boolean Matrix Factorization (BMF)} if $f^\top \circ f$, restricted to its image, is the identity map on $\cC$. In this case we say \emph{$\cD$ is a factor of $\cC$}.
        \end{defn}

        \begin{lem}\label{lem:adjointiscanonical}
            The map $f^\top$ is the Galois adjoint of the canonical lift of $f$, restricted to $\cD$.
        \end{lem}
        \begin{proof}
            The canonical lift of $f$ is represented by
            \[ t = (\sqrt{f^{-1}(\tk_\cD(1))},\dots,\sqrt{f^{-1}(\tk_\cD(k))}). \]
            By \cref{def:galoisconnectionofH} (see \eqref{eq:galoisF}), the adjoint of of $f_t$ is defined by sending $d \subseteq [r]$ to the union of the representative sets indexed by $d$, which is exactly the expression in \eqref{eq:ftop}.
        \end{proof}
        \noindent Note that \cref{lem:useHtogetfactorization,lem:adjointiscanonical} justify referring to a morphism $f: \bC \to \bD$ satisfying $f^\top \circ f = id_\bC$ as a Boolean matrix factorization: given such a morphism $f$, the matrix $\bH$ representing the canonical lift of $f$ is exactly the matrix such that $\bC = G_\bH(\bC)\bH = (\bC : \bH)\bH$.
        In light of \cref{lem:Hmaximalfactorizationmorphism}, given a factorization $\bC = \bV\bH$, we have $\bC = (\bV\bH: \bH)\bH$.
        Thus, we can use these notions of factorization interchangeably.

        The property of being a BMF is, in some sense, intrinsic to the partial order $\pcode$.
        That is, if $f: \cC \to \cD$ is a BMF, then there is a map $f': \rho(\cC) \to \rho(\cD)$ that is also a BMF (recall from \cref{rmk:reductionmorphism} $\rho$ is the reduction morphism).
        We will show this in two stages, first by showing you can reduce $\cC$ while maintaining a BMF, and then by reducing $\cD$.
        \begin{lem}\label{lem:BMFreducesource}
            Let $f: \cC \subseteq 2^{[n]} \to \cD \subseteq2^{[k]}$ be a BMF. Then there is a BMF $\rho(\cC) \to \cD$.
        \end{lem}
        \begin{proof}
            Without loss of generality, suppose $n$ is redundant to $\sigma$ in $\cC$.
            Let $t = (\tau_1,\dots,\tau_k)$ be the canonical representative of $f$, and set $t' = (\tau_1\cap[n-1],\dots,\tau_k\cap[n-1])$.

            First, we claim that $f_t(c) = f_{t'}(c\cap[n-1])$, which requires $\tau_j \subseteq c$ if and only if $\tau_j\cap[n-1] \subseteq c\cap[n-1]$.
            We prove the converse direction.
            If $n\not \in \tau_j$ then \[ \tau_j=\tau_j \cap [n-1] \subseteq c \cap [n-1]\subseteq c \] yields $\tau_j \subseteq c$.
            Suppose now that $n \in \tau_j$.
            By assumption, $n$ is redundant to some $\sigma \subseteq [n-1]$, so we must have $\sigma\subseteq\tau_j$.
            But then we also have $\sigma \subseteq \tau_j \cap [n-1] \subseteq c\cap[n-1]$, which implies $n \in c$.
            Then \[ \tau_j = \tau_j\cap[n-1]\cup n \subseteq c\cap [n-1] \cup n = c \] yields $\tau_j \subseteq c$ as desired, completing the proof of the claim.

            Direct evaluation then shows
            \begin{align*}
                f_{t'}^\top(f_{t'}(c \cap [n-1])) & = f_{t'}^\top(f_t(c))\\
                    & = \bigcup_{j \in f_t(c)} \tau_j \cap [n-1]\\
                    & = \left(\bigcup_{j\in f_t(c)} \tau_j\right) \cap [n-1]\\
                    & = f_t^\top(f_t(c)) \cap [n-1] = c \cap [n-1].
            \end{align*}
            Therefore, $f_{t'}$ is a BMF from $\cC$ with neuron $n$ deleted to $\cD$. The statement of the theorem follows by permuting neurons in $\cC$ if necessary and inducting on the number of redundant neurons in $\cC$.
        \end{proof}

        \begin{lem}\label{lem:redundancyinimage}
            Let $f: \cC \subseteq2^{[n]} \to \cD\subseteq2^{[k]}$ be a surjective morphism of codes, and suppose $k$ is redundant in $\cD$.
            For any representative $t = (\tau_1,\dots,\tau_k)$ of $f$, the tuple $t' = (\tau_1,\dots,\tau_{k-1})$ represents composition of $f$ with deletion of neuron $k$ in $\cD$.
            In particular, the image of $f_{t'}$ is isomorphic to the image of $f$.
        \end{lem}
        \begin{proof}
            Note that $f(\sigma) = f_t(\sigma) = \{j \in [k] : \tau_j \subseteq \sigma\}$ and so $f_{t'}(\sigma) = f_t(\sigma) \setminus k$. Since $k$ is redundant in $\cD$, deletion of $k$ is an isomorphism.
        \end{proof}

        \begin{thm}\label{lem:bmfredundancy}
            Suppose $f: \cC \to \cD$ is a BMF with $\cC \subseteq 2^{[n]}$, $\cD \subseteq 2^{[k]}$.
            If $\cC$ is reduced, 
            then there is a BMF $\cC \to \rho(\cD)$. 
        \end{thm}
        \begin{proof}
            We will show that, if $k$ is redundant to $\sigma \subseteq[k-1]$ in $\cD$, then composing $f$ with the deletion of neuron $k$ is a BMF.

            For each $i \in [k]$, set $T_i = f^{-1}(\tk_\cD(i))$, and set $t = (\sqrt{T_1},\dots,\sqrt{T_k})$ a representative of $f$. Our goal is to show that the map represented by $t' = (\sqrt{T_1},\dots,\sqrt{T_{k-1}})$ is also a Boolean matrix factorization. Since $k$ is redundant in $\cD$, from \cref{lem:redundancyinimage} we know that the image of  $f_{t'}$ is exactly $\cD$ with $k$ deleted, so all that remains to show is that $f_{t'}^\top (f_{t'}(c)) = c$ for all $c \in \cC$.
            
            First, we claim that $f_{t'}^\top(f_{t'}(c)) = f^\top(f(c)\setminus k)$.
            That $f_{t'}(c) = f(c) \setminus k$ follows from the definition of $f_{t'}$.
            Writing $f_{t'}^\top(d) = \bigcup_{j\in d} \sqrt{T_j}$ we see that this agrees with $f^\top(d)$ when $k \notin d$.
            
            Consider some $c \in \cC$ and let $d = f(c)$.
            If $c \notin T_k$, then $k \notin d$, so from above we have $f_{t'}(c) = f(c)$ and $f_{t'}^\top (f_{t'}(c)) = c$.

            Now consider the case $c \in T_k$, so $k \in d$.
            Since $k$ is redundant to $\sigma$, $d \in \tk_\cD(\sigma)$ and so we can write $d$ as the disjoint union $d = d' \cup \sigma \cup k$. Then
            \begin{align*}
                f^\top(d) = \bigcup_{j\in d'} \sqrt{T_j} \cup \bigcup_{j\in\sigma} \sqrt{T_j} \cup \sqrt{T_k}\tag{*}\label{eq:needtoreducefactor}
            \end{align*}
            We will show that $\sqrt{T_k}$ is unnecessary in this union.
            
            Since $k$ is redundant to $\sigma$, we have $\tk_\cD(k) = \tk_\cD(\sigma) = \bigcap_{j\in\sigma} \tk_\cD(j)$.
            Applying $f^{-1}$ to both sides we have $T_k = \bigcap_{j\in\sigma} T_j$ and so, taking the root of both sides and applying \cref{lem:rootsofintersections},
            \[ \sqrt{T_k} = \sqrt{\bigcap_{j\in\sigma} T_j} \supseteq \bigcup_{j\in\sigma} \sqrt{T_j}. \]
            If $\sqrt{T_k} = \bigcup_{j\in\sigma} \sqrt{T_j}$, then clearly $\sqrt{T_k}$ is unnecessary in \eqref{eq:needtoreducefactor}.

            Suppose, on the other hand, there is some $\ell \in \sqrt{T_k} \setminus \bigcup_{j\in\sigma}\sqrt{T_j}$. If $\tk_\cC(\ell) = \tk_\cC(\sqrt{T_k}) =T_k$
            then $\ell$ is redundant to $\bigcup_{j\in\sigma} \sqrt{T_j}$
            because
                \[ \tk_\cC\Bigl(\bigcup_{j\in\sigma} \sqrt{T_j}\Bigr) = \bigcap_{j\in\sigma} \tk_\cC(\sqrt{T_j}) = \bigcap_{j\in\sigma} T_j = T_k = \tk_\cC(\ell). \]
            Since $\cC$ is reduced, therefore, we must have $\tk_\cC(\ell) \supsetneq T_k$.
            Consider $\alpha \in \tk_\cC(\ell) \setminus T_k$. Since $\ell \in \alpha$, $k \notin f(\alpha)$, and $f^\top(f(\alpha)) = \alpha$, we see from \eqref{eq:needtoreducefactor} that there must be some $j \in d'$ such that $\ell \in \sqrt{T_j}$. Once again, we conclude that $\sqrt{T_k}$ is unnecessary in \eqref{eq:needtoreducefactor}.

            Therefore, $f_{t'}$ is a BMF.
            The statement of the lemma follows by permuting neurons in $\cD$ if necessary and inducting on the number of redundant neurons in $\cD$.
        \end{proof}

    \subsection{Adjoints of Covering Maps}\label{sec:adjointsofcovering}
        A code $\cC$ \emph{covers} a code $\cD$ if $\cC \geq \cD$ but they are not isomorphic, and for any $\cE$ satisfying $\cC \geq \cE \geq \cD$, we have $\cC \cong \cE$ or $\cE \cong \cD$.
        We now show how to determine which factors of a code $\cC$ are covered by it.
        
        The construction below appears in \cite{jeffs2019sunflowers}*{Definition 3.19}; 
        we explicitly index the trunks to simplify the development that follows.

        \begin{defn}\label{def:excessivecovering}
            Let $\cC \subseteq 2^{[n]}$ and $i \in [n]$. For $j = 1,\dots, 2n$ define the following trunks:
            \[ T_j = \begin{cases} \tk_\cC(j) & j\leq n \text{ and } \tk_\cC(j) \neq \tk_\cC(i)\\
                \tk_\cC(j-n) \cap \tk_\cC(i) & j > n \text{ and } \tk_\cC(j-n) \cap \tk_\cC(i) \neq \tk_\cC(i)\\
                \varnothing & \text{o.w.}
            \end{cases} \]
            Let $f_i$ be the morphism defined by these trunks, and $\cC^{(i)}$ the image of $f_i$. This is the \emph{$i$-th covering map of $\cC$} and $\cC^{(i)}$ is its \emph{$i$-th covered code.}
            
            When $\cC$ is reduced, the explicit formula for $f_i$ is
                \begin{align}
                    f_i(c) = (c\setminus i) \cup \Bigl\{j+n : ij\subseteq c, \tk_\cC(i) \cap \tk_\cC(j) \neq \tk_\cC(i)\Bigr\} \label{eq:coveringformula}
                \end{align}
        \end{defn}
        \begin{prop}[\cite{jeffs2019sunflowers}*{Theorem 3.20}]
            If $\cC$ is reduced, then for any $i$, $\cC$ covers $f_i(\cC)$.
        \end{prop}

        \begin{rmk}
            The above definition will result in at least two trivial neurons in the image, since $T_i$ and $T_{n+i}$ are both empty.
            However, the formula given above simplifies the ensuing discussion.
            In light of \cref{lem:bmfredundancy}, if $\cC$ is reduced it does not matter what representative we pick for $\cC^{(i)}$ -- if any representative is a factor of $\cC$, then so is $\rho(\cC^{(i)})$.
        \end{rmk}

        \begin{lem}\label{lem:norootinkernel}
            Let $\cC \subseteq 2^{[n]}$ and suppose $i \in [n]$ is not redundant.
            Let $f_i$ be the $i$-th covering map of $\cC$. Then $\sqrt[\cC]{i} \notin \im f_i^\top \circ f_i$.
        \end{lem}
        \begin{proof}
            From \eqref{eq:coveringformula}, we have
            \begin{align}
                f_i^\top(f_i(c)) = (c \setminus i) \cup \bigcup_{\substack{ij\subseteq c \\\tk_\cC(ij) \neq \tk_\cC(i)}} \sqrt[\cC]{ij}. \label{eq:ftopf}
            \end{align}
            Suppose $f_i^\top(f_i(c)) \subseteq \sqrt[\cC]{i}$.
            By \cref{lem:trunkrootcontainment}, we know that any $j$ which satisfies the condition of the big union in \eqref{eq:ftopf} must have $\sqrt[\cC]{ij} \supsetneq \sqrt[\cC]{i}$, and so
            \[ \Bigl\{ij \subseteq c : \tk_\cC(j) \cap \tk_\cC(i) \neq \tk_\cC(i)\Bigr\} = \varnothing. \]
            If this set were nonempty then some $j \in \cC$ would contribute $\sqrt[\cC]{\{i,j\}}$ to the union in \eqref{eq:ftopf}.
            But then $f_i^\top(f_i(c)) \supseteq \sqrt[\cC]{ij} \supsetneq \sqrt[\cC]{i}$, contradicting the supposition $f_i^\top(f_i(c)) \subseteq \sqrt[\cC]{i}$.
            So it must be empty and this means that $f_i^\top(f(c)) = c\setminus i \subsetneq \sqrt[\cC]{i}$.
            Therefore, $\sqrt[\cC]{i} \notin \im f_i^\top \circ f_i$.
        \end{proof}

        \noindent Thus, we have a necessary condition for a covered code to be a factor:

        \begin{cor}\label{cor:BMFimpliesfree}
            Let $\cC \subseteq 2^{[n]}$ and $f_i$ its $i$-th covering map. If $f_i$ is a BMF, then $\sqrt[\cC]{i} \notin \cC$ (that is, $i$ is free).
        \end{cor}
        \begin{proof}
            Suppose $f_i$ is a BMF, that is, $f_i^\top\circ f_i = id_\cC$. By \cref{lem:norootinkernel}, $\sqrt[\cC]{i} \notin \im f_i^\top \circ f_i = \cC$.
        \end{proof}

        Note that \cref{eq:ftopf} implies that covering maps are ``nearly'' injective:
        \begin{lem}\label{lem:squeezeandcollision}
            Let $\cC$ be a reduced code and $f_i$ its $i$-th covering map. Then we have the following:
            \begin{enumerate}
                \item\label{lem:squeeze} (Squeeze lemma) For any $\sigma \in \cC$, we have $\sigma \setminus i \subseteq f_i^\top(f_i(\sigma)) \subseteq \sigma.$
                \item\label{lem:collision} (Weak collision lemma) If $\sigma \neq \tau$ are distinct codewords such that $f_i(\sigma) = f_i(\tau)$, then $i \notin \sigma$ and $\tau = \sigma\cup i$
                \item\label{lem:strongcollision} (Strong collision lemma) If $\sigma\neq \tau$, then $f_i(\sigma) = f_i(\tau)$ if and only if $\sigma = \sqrt[\cC]{i} \setminus i$ and $\tau = \sqrt[\cC]{i}$.
            \end{enumerate}
        \end{lem}
        \begin{proof}
            \cref{lem:squeeze} follows from \cref{eq:ftopf}.
            \cref{lem:collision} follows because $f_i(\sigma) = f_i(\tau)$ requires $\sigma\setminus i = \tau \setminus i$.

            \cref{lem:strongcollision}.
            $(\Leftarrow)$ $f_i(\sqrt[\cC]{i}\setminus i) = f_i(\sqrt[\cC]{i}) = \sqrt[\cC]{i}\setminus i$ from all the preceding definitions and lemmas.
            
            $(\Rightarrow)$ Suppose $f_i(\sigma) = f_i(\tau)$ with $\sigma\neq \tau$.
            By \cref{lem:collision}, we must have (without loss of generality) $\tau = \sigma\cup i$, with $i\notin \sigma$. Note that since $\tau \in \cC$ and $i \in \tau$, we must have $\sqrt[\cC]{i} \subseteq \tau$.
            It remains to show that $\tau \subseteq \sqrt[\cC]{i}$ to complete the proof.
            
            Since $i \notin \sigma$, we have $f_i(\sigma) = \sigma$ and so $f_i(\tau) = \sigma$.
            This means we have
                \begin{align*}
                    \{j + n : ij\subseteq \tau, \sqrt[\cC]{ij} \supsetneq \sqrt[\cC]{i}\} = \varnothing
                    \quad \text{which implies}\quad \{ j \in \sigma : \sqrt[\cC]{ij} \supsetneq \sqrt[\cC]{i}\} = \varnothing
                \end{align*}
                which means that for all $j \in \sigma$, we have $\sqrt[\cC]{ij} = \sqrt[\cC]{i}$. Therefore, $\sigma \subseteq \sqrt[\cC]{i}$.
                In fact, the containment must be strict since $i \notin \sigma$, and so in particular we must have $\tau = \sigma \cup i \subseteq \sqrt[\cC]{i}$, completing the proof.
        \end{proof}

        We now show that the condition in \cref{cor:BMFimpliesfree} is sufficient.

        \begin{thm}\label{thm:coveringBMF}
            Let $\cC \subseteq 2^{[n]}$ be a reduced code and $f_i: \cC \to \cC^{(i)}$ the $i$-th covering map.
            Then $f_i$ is a BMF if and only if $i$ is free.
        \end{thm}
        
        \begin{proof}
            The forward implication is \cref{cor:BMFimpliesfree}.
            
            For the converse, suppose $i$ is free, that is, $\sqrt[\cC]{i} \notin \cC$. We want to show $f_i^\top(f_i(\sigma)) = \sigma$ for all $\sigma \in \cC$. So, consider some codeword $\sigma \in \cC$. If $i \notin \sigma$, then $\sigma\setminus i = \sigma$ and so by \cref{lem:squeezeandcollision}(\ref{lem:squeeze}) we have $f_i^\top(f_i(\sigma)) = \sigma$.
            
            Now suppose $i \in \sigma$.
            Since $i$ is free, we have a strict containment $\sqrt[\cC]{i} \subsetneq \sigma$, and so there is some $j \in \sigma \setminus \sqrt[\cC]{i}$.
            For this $j$, $\sqrt[\cC]{ij} \supsetneq \sqrt[\cC]{i}$; applying \cref{lem:trunkrootcontainment} we see the set
                \[\{ij \subseteq \sigma : \tk_\cC(j) \cap \tk_\cC(i) \neq \tk_\cC(i)\}\]
            is nonempty. Therefore, the big union in 
                \[ f_i^\top(f_i(\sigma)) = (\sigma \setminus i) \cup \bigcup_{\substack{ij\subseteq \sigma \\\tk_\cC(ij) \neq \tk_\cC(i)}} \sqrt[\cC]{ij} \]
            has at least one term and so $i \in f_i^\top(f_i(\sigma))$. By \cref{lem:squeezeandcollision}(\ref{lem:squeeze}), this means $f_i^\top(f_i(\sigma)) = \sigma$, completing the proof.
        \end{proof}

    \subsection{Defect of a code}\label{sec:defect}
        \Cref{lem:squeezeandcollision} actually shows that covering maps are very rigid, in some sense, 
        as composition with the adjoint must send all codewords nearly back to themselves. 
        In fact, the image of a covering map has cardinality equal to that of its domain or exactly one less. We will use this to define a weak bigrading on $\pcode$.

        For a code $\cC$, let $T(\cC)$ denote the set of nonempty trunks of $\cC$, and let $t(\cC) = |T(\cC)|$ count the number of nonempty trunks of $\cC$.
        This is a rank function on $\pcode$ (see \cite{jeffs2019sunflowers}*{Section 3}, in particular Remark 3.23): $\cC$ covers $\cD$ if and only if $\cC \geq \cD$ and $t(\cC) = t(\cD) + 1$.
        Every codeword of $\cC$ defines a distinct trunk, so we have the following lemma.

        \begin{lem}\label{lem:trunkinjection}
            For any code $\cC$, $|\cC| \leq t(\cC)$.
        \end{lem}
        \begin{proof}
            Consider $\tk_\cC$ as a map $\cC \to T(\cC)$. It is injective: if $\sigma \neq \tau$ then at least one of the two is not in the trunk of the other so $\tk_\cC(\sigma) \neq \tk_\cC(\tau)$.
        \end{proof}

        \begin{defn}\label{def:defect}
            For $\cC$ a code, define the \emph{defect} of $\cC$ by $d(\cC) = t(\cC) - |\cC|$.
        \end{defn}
        From \cref{lem:trunkinjection}, we have $d(\cC) \geq 0$ for all codes.
        In fact, $d(\cC) = 0$ if and only if $\cC$ is intersection-complete as the next proposition shows.
        \begin{prop}[\cite{jeffs2025coveringrelationsposetcombinatorial}*{Proposition 33}]\label{lem:intcompletetrunks}
            Restricted to $\cC$, $\tk_\cC$ is a bijection if and only if $\cC$ is intersection-complete.
        \end{prop}
        \begin{proof}
            $(\Rightarrow)$ If $\cC$ is intersection-complete, then $\sqrt{T} \in \cC$ for any nonempty trunk $T \subseteq \cC$. 
            Thus $\tk_\cC(\sqrt{T}) = T$ meaning $\tk_\cC$ is surjective and therefore bijective.
            
            $(\Leftarrow)$ Suppose $\sigma, \tau \in \cC$. Let $T = \tk_\cC(\sigma \cap \tau)$. By assumption, $T = \tk_\cC(c)$ for some $c \in \cC$.
            By \cref{lem:misctrunks}~(\ref{rootedtrunk}), we have $c = \sqrt{T}$. But then $\sigma \cap \tau \subseteq c$ by \cref{lem:misctrunks}~(\ref{insideownroot}).
            
            On the other hand, both $\sigma \in T$ and $\tau \in T$ so $c =\sqrt{T} \subseteq \sigma \cap \tau$. Therefore $c = \sigma \cap \tau$ and so $\sigma \cap \tau \in \cC$.
        \end{proof}
        Defect is an isomorphism invariant of codes.

        \begin{lem}
            Suppose $\cC \cong \cD$ are isomorphic codes. Then $d(\cC) = d(\cD)$.
        \end{lem}
        \begin{proof}
            An isomorphism is a bijection on codewords and on trunks \cite{jeffs2019sunflowers}. Therefore $|\cC| = |\cD|$ and $t(\cC) = t(\cD)$, from which it follows $d(\cC) = d(\cD)$.
        \end{proof}
        Defect is order-preserving on $\pcode$. In fact, we show something stronger.

        \begin{thm}\label{thm:bigrading}
            Let $\cC$ be a code and $i$ not redundant in $\cC$. Then 
            \[d(\cC) - 1 \leq d(\cC^{(i)}) \leq d(\cC).\]
        \end{thm}
        
        \begin{proof}
            The strong collision lemma (\cref{lem:squeezeandcollision}(\ref{lem:strongcollision})) states that a covering map is either injective (so $|\cC^{(i)}| = |\cC|$), or the image has cardinality one less than the domain (so $|\cC^{(i)}| = |\cC| - 1$).
            Since it's a covering map, $t(\cC^{(i)}) = t(\cC) - 1$.
            Putting this together, we have
            \begin{align*}
                d(\cC^{(i)}) & = t(\cC^{(i)}) -|\cC^{(i)}|\\
                    & = (t(\cC) - 1) - (|\cC| - x)\\
                    & = d(\cC) - 1 + x
            \end{align*}
            where $x = 0$ or $1$.
            Therefore, $d(\cC) - 1 \leq d(\cC^{(i)}) \leq d(\cC)$.
        \end{proof}
        \begin{cor}
            Let $f: \cC \to \cD$ be a surjective morphism of codes.
            Then $d(\cC) \geq d(\cD)$.
        \end{cor}

        Every isomorphism class of codes can thus be assigned a pair of coordinates $(d(\cC), t(\cC))$; covering maps change the coordinates by either  $(0,-1)$ or $(-1,-1)$. Cardinality is only preserved in the $(-1,-1)$ direction, which means that any bijectiive morphism of codes (other than an isomorphism) must decrease defect.

        \begin{cor}\label{cor:BMFdefect}
            If $f: \cC \to \cD$ is a BMF, then either it is an isomorphism or it strictly decreases defect.
        \end{cor}
        \begin{proof}
            Suppose $f:\cC\to\cD$ is a BMF, so in particular it is a surjective morphism of codes.
            If $t(\cC) = t(\cD)$, then $f$ is an isomorphism \cite{jeffs2019sunflowers}*{Proposition 3.16}.            
            Otherwise, we must have $t(\cC) > t(\cD)$. Since $|\cC| = |\cD|$ we have $d(\cC) = t(\cC) - |\cC| > t(\cD) - |\cD| = d(\cD)$.
        \end{proof}

    \subsection{Another Galois Connection}
        We note that \cref{lem:intcompletetrunks} is suggestive: given a fixed code $\cC$, the map $\sigma \mapsto \tk_\cC(\sigma)$ defines a map from $2^{[n]}$ to its power set. In the oppposite direction we have the map $S \mapsto \sqrt{S}$. In fact, these maps form a Galois connection as well.
        \begin{lem}
            Let $\cC\subseteq 2^{[n]}$ be fixed. Then the maps $G: S \mapsto \sqrt{S}$ and $F: \sigma \mapsto \tk_\cC(\sigma)$ are the upper and lower adjoints, respectively, in a Galois connection between $2^\cC$ and $2^{[n]}$.
        \end{lem}
        \begin{proof}
            Suppose $\tk_\cC(\sigma) \subseteq S$. Then every codeword in $S$ contains $\sigma$, and so $\sigma \subseteq \sqrt{S}$.

            Suppose $\sigma \subseteq \sqrt{S}$. Then $\sigma \subseteq s$ for all $s \in S$, and so each $s\in S$ is also in $\tk_\cC(\sigma)$.
        \end{proof}
        \Cref{lem:intcompletetrunks} then gives the condition under which these maps, restricted to $\cC$ and $T(\cC)$, are mutual inverses (namely, when $\cC$ is intersection-complete, which is true if and only if $d(\cC) = 0$).

\section{Boolean rank}\label{sec:booleanrank}

    We now show an application of our results to computing the Boolean rank of a matrix.
    A rank $r$ factoriziation of a matrix $\bC \in \B^{m\times n}$ is an expression of $\bC$ as the Boolean product of matrices $\bC = \bV\bH$ with $\bV\in \B^{m\times r}$ and $\bH \in \B^{r\times n}$. The Boolean rank of $\bC$, denoted $\brank(\bC)$, is the minimum $r$ for which there is a rank $r$ factorization of $\bC$.
    For a code $\cC \subseteq 2^{[n]}$, we define $\brank(\cC)$ as the Boolean rank of any matrix representing $\cC$.
    The following small example shows that Boolean rank is not an isomorphism invariant of combinatorial codes.
    \begin{ex}\label{ex:isomorphicnotbrank}
        The following matrices represent isomorphic codes in $\code$:
        \[ \bC = \begin{bmatrix}
            0&0\\
            1&0\\
            0&1\\
            1&1
        \end{bmatrix} \qquad
        \bC' = \begin{bmatrix}
            0&0&0\\
            1&0&0\\
            0&1&0\\
            1&1&1
        \end{bmatrix}\]
        In $\bC'$, neuron $3$ is redundant to $\{1,2\}$; deleting this redundant neuron yields $\bC$. By inspection, $\brank(\bC) = 2$ while $\brank(\bC') = 3$.
    \end{ex}
    
    However, it is true that the reduced representative(s) of an isomorphism class have the smallest Boolean rank of codes in that class.
    \begin{lem}\label{lem:reducedminimumrank}
        For any code $\bC$, $\brank(\rho(\bC)) \leq \brank(\bC)$.
    \end{lem}
    \begin{proof}
        Suppose $\bC = \bV\bH$ is a rank $r = \brank(\bC)$ factorization of $\bC$. By \cref{lem:rhoisprojection} there is a $\bP$ such that
        \[ \rho(\bC) = \bC\bP = (\bV\bH)\bP = \bV(\bH\bP)\]
        which is a rank $r$ factorization of $\rho(\bC)$.
    \end{proof}

    Taken together, our results suggest a graph search algorithm for factoring a reduced Boolean matrix: First, enumerate the free neurons of code $\bC$. For each free neuron $i$, compute a reduced representative of covered code $\bC^{(i)}$, and then recursively search. The graph we are searching is the Hasse diagram of $\pcode$, restricted to the edges which represent matrix factorizations, part of which is illustrated in \cref{fig:pcode3BMF}.
    The following lemma shows that determining if a neuron $i$ is redundant can be done in time linear in the number of entries in the matrix representing $\cC$ (and thus, for a fixed $n$, linear in the size of the code).

    \begin{lem}[Redundancy Lemma]\label{lem:redundancylemma}
        If $i$ is redundant to $\sigma$, then $i$ is redundant to $\tau$ for any $\tau \in \gens(\tk_\cC(\sigma))$ with $\tau \subseteq \sigma$.
        In particular, neuron $i$ is redundant if and only if $\tk_\cC(i) = \tk_\cC(\sqrt[\cC]{i} \setminus i)$.
    \end{lem}
    
    \begin{proof}
        The meaning of $\tau \in \gens(\tk_\cC(\sigma))$ is that $\tk_\cC(\tau) = \tk_\cC(\sigma) = \tk_\cC(i)$. If $\tau \subseteq \sigma \subseteq [n] \setminus i$, then $i$ is redundant to $\tau$.
        
        Clearly, if $\tk_\cC(i) = \tk_\cC(\sqrt[\cC]{i} \setminus i)$ then $i$ is redundant, so the only thing we need to check is the converse. Suppose $i$ is redundant to $\sigma$, with $\sigma \subseteq [n] \setminus i$. Then we have $\tk_\cC(i) = \tk_\cC(\sigma) = \tk_\cC(\sigma\cup i)$, and so in particular $\sqrt[\cC]{i} \setminus i \in \gens(\tk_\cC(\sigma))$. 
    \end{proof}

    Redundant neurons also bound Boolean rank of a matrix relative to the Boolean rank of a reduced representative of that matrix.
    \begin{lem}\label{lem:redundancybound}
        Suppose $\bC$ is a code with $k$ redundant neurons. Then $\brank(\bC) \leq \brank(\rho(\bC)) + k$.
    \end{lem}
    \begin{proof}
        Without loss of generality, assume neurons $n-k+1,\dots,n$ are redundant in $\bC$. Suppose $\rho(\bC) = \bV\bH$ is a rank $r = \brank(\rho(\bC))$ factorization of $\rho(\bC)$. Let $\bC'$ be the last $k$ columns of $\bC$, and set
        \[\bV' = \begin{bmatrix}
            \bV & \bC'
        \end{bmatrix}
        \quad \text{and}\quad
        \bH' = \begin{bmatrix}
            \bH & \mathbf{0}\\
            \mathbf{0} & \bI
        \end{bmatrix}\]
        where $\bI$ is a $k \times k$ identity matrix. Then we have $\bC = \bV'\bH'$ is a rank $r + k$ factorization of $\bC$.
    \end{proof}

        Our results also show that intersection-complete codes must be full Boolean rank.

        \begin{thm}\label{thm:intcompletefullrank}
            If $\cC = \{\sigma_1,\dots,\sigma_m\} \subseteq 2^{[n]}$ is intersection-complete, then $\brank(\cC) = \min(m,n)$
        \end{thm}
        \begin{proof}
            By \cref{cor:BMFdefect}, if $\cC \to \cC'$ is a Boolean matrix factorization, we must have the strict inequality $d(\cC') < d(\cC)$.
            However, since $\cC$ is intersection-complete, $d(\cC) = 0$.
        \end{proof}

        \noindent Determining if $\bC$ is intersection-complete can be accomplished by inspecting the generators of $I_\bC$, per \cref{lem:intgens}.

        Finally, we can use the neural ring to bound Boolean rank from below.
        We say $R_\bC$ is \emph{generated by $r$ monomials} if there is a surjective monomial map $\F_2[x_1,\dots,x_r] \to R_\bC$. The smallest $r$ such that $R_\bC$ is generated by $r$ monomials is the \emph{monomial rank of $R_\bC$}, denoted $\mrank(R_\bC)$.

        \begin{thm}\label{thm:monomialrankbound}
            For any code $\bC$, $\mrank(\bC) \leq \brank(\bC)$.
        \end{thm}
        \begin{proof}
            Let $\bC = \bV\bH$ be a rank-$r$ factorization of $\bC$ such that $\bV$ is $\bH$-maximal. By \cref{lem:Hmaximalfactorizationmorphism}, the map $\bC \to \bV$ is a bijective morphism of codes, so its pullback $R_\bV \to R_\bC$ is surjective. Thus composition with the natural map $\F[x_1,\dots,x_r] \to R_\bV$ yields a surjection $\F[x_1,\dots,x_r] \to R_\bC$.
        \end{proof}


\section{Discussion and Open Questions}\label{sec:conclusion}

Our work places codes and morphisms of codes on the same footing, by representing both with binary matrices.
A combinatorial code can be seen as an equivalence class of matrices, where matrices are equivalent if the set of supports of the rows of the two matrices produce isomorphic combinatorial codes.
However, different representative matrices of morphisms typically yield distinct combinatorial codes (recall \cref{ex:nonuniquemonomialmap}), so morphisms of codes induce a different partition on the set of binary matrices.
Future work might investigate what properties of matrices are invariant in these equivalence classes.

Our work complements the recent work of Jeffs and Trang~\cite{jeffs2025coveringrelationsposetcombinatorial}. They study the codes that cover a given code, giving the complementary operation to \cref{def:excessivecovering}. A natural question to ask is whether there is an analogue of \cref{thm:coveringBMF} for their construction: That is, which codes covering a given code can be obtained by matrix multiplication?

They also showed that passing to the intersection completion is an endofunctor on the category of codes, and that this functor preserves covering relations in $\pcode$~\cite{jeffs2025coveringrelationsposetcombinatorial}*{Lemmas 13 and 14}. 
We can describe this pictorially by saying that every edge in \cref{fig:pcode3} can be ``translated'' to the left, giving an edge of the Hasse diagram in the $d = 0$ column.
A covering relation between intersection-complete codes cannot be a BMF (\cref{cor:BMFdefect}). However, is there any property of the codes or the morphism which guarantees that one of the edges that translates to it is a BMF?

The problem of finding a low-rank Boolean matrix factorization is equivalent to a number of problems, such as finding a small cover of a bipartite graph by bicliques or factorizing a general relation between two sets.
We offer one more equivalent formulation: finding the minimal number of variables in a polynomial ring that a given monomial map factors through (see \cref{sec:pullback}).

We conclude with some specific questions and conjectures that have arisen over the course of our work.

\subsection{Interest in intersection completion.}
    In a preprint version of \cite{curto2019neuralringhomomorphisms}, the authors present a streaming algorithm to construct the canonical form of a code $\cC$. Given an ordering of the codewords $\sigma_1,\dots,\sigma_m$, for each $i = 1,\dots,m$ it constructs the canonical form of the subset $\cC_i = \{\sigma_1,\dots,\sigma_i\} \subseteq\cC$. 
    Thus, if the elements that violate \cref{lem:intgens} (respectively, \cref{lem:uniongens}) are discarded at each intermediate step, we obtain at each step the canonical form of $\widehat{\cC_i}$ (resp., $\check{\cC}_i$), and the output of the algorithm is the canonical form of $\widehat{\cC}$ (resp., $\check{\cC}$).
    The size of the canonical form can grow very large for arbitrary codes, so a natural question is how quickly the size of the canonical form of an intersection-complete (union-complete) code can grow.

    \begin{question}
        What is the largest number of elements in the canonical form of an intersection complete or union-complete code on $n$ neurons?
    \end{question}
    \noindent This is posed as an extremal problem, but there is an accompanying enumerative problem: For a given size of canonical form, how many codes on $n$ neurons achieve a canonical form of that size? 

\subsection{Properties of Defect}
    Let $\lambda(\cC)$ denote the miniminum neuron number of a code, the smallest $n$ such that $\cC$ is isomorphic to a code $\cC' \subseteq 2^{[n]}$.
    \Cref{fig:pcode3} illustrates the down-set of $\pcode$ generated by all codes with $\lambda = 3$.
    Surprisingly, this includes some codes with $\lambda = 4$, showing that $\lambda$ is not order-preserving on $\pcode$.
    \begin{question}
        What is $\max_{\cC \in \code} \{\lambda(\cC^{(i)}) - \lambda(\cC)\}$?
    \end{question}
    From \cref{def:excessivecovering} we see that $\lambda(\cC^{(i)}) \leq 2\lambda(\cC) - 2$. Does any code achieve this bound? Moreover, we see from \cref{fig:pcode3} that $24$ of the $82$ codes in the down-set generated by $\lambda = 3$ have $\lambda = 4$. What is the largest number of codes of length $\ell + 1$ in the down-set generated by codes of length $\ell$?

\subsection{Relevance of redundancy}
    Adding a repeated row to a matrix does not change the corresponding code and does not change the Boolean rank.
    Adding a repeated column or a column of all $1$s likewise does not change the Boolean rank, but it does introduce a redundant neuron to the code. 
    However, \cref{ex:isomorphicnotbrank} shows that it is possible to increase the Boolean rank of a matrix by adding a column equal to a product of other columns but not duplicating an existing column.
    Thus, we can distinguish between \emph{simply redundant} neurons (a neuron which is redundant to a singleton set or the empty set) and \emph{non-simple redundancies} (a neuron which is redundant but not simply redundant). 
    The bound in \cref{lem:redundancybound} could therefore be improved by replacing $k$ by the number of simply redundant neurons, but
    we conjecture that \cref{lem:reducedminimumrank,lem:redundancybound} can be made exact.

    \begin{conj}
        Let $\bC$ be a binary matrix with $k$ redundant neurons of which $\ell \leq k$ are not simply redundant. Then $\brank(\bC) = \brank(\rho(\bC)) + \ell$.
    \end{conj}

    A related question is what is the largest $\ell$ can be. We note that adding non-simple redundancies to a code amounts to adding intersections of the supports of columns of a representative matrix, and so there is an upper bound of $\ell \leq 2^{\lambda(\bC)} - 1$, as this counts all nonempty subsets of distinct non-redundant columns. We expect this bound is rarely achieved among codes, and any progress on bounding $\ell$ in turn bounds the gap between $\brank(\bC)$ and $\brank(\rho(\bC))$.

\bibliographystyle{amsalpha}
\bibliography{BMFandNeuralRings}

@book {Bergman,
    AUTHOR = {Bergman, George M.},
     TITLE = {An invitation to general algebra and universal constructions},
    SERIES = {Universitext},
   EDITION = {Second},
 PUBLISHER = {Springer, Cham},
      YEAR = {2015},
     PAGES = {x+572},
      ISBN = {978-3-319-11477-4; 978-3-319-11478-1},
   MRCLASS = {08-01 (00A05 18Axx 20Fxx)},
  MRNUMBER = {3309721},
       DOI = {10.1007/978-3-319-11478-1},
       URL = {https://doi-org.libproxy.unl.edu/10.1007/978-3-319-11478-1},
}

@article{cruz2019open,
	author = {Cruz, Joshua and Giusti, Chad and Itskov, Vladimir and Kronholm, Bill},
	date = {2019/03/01},
	doi = {10.1007/s00454-018-00050-1},
	id = {Cruz2019},
	isbn = {1432-0444},
	journal = {Discrete \& Computational Geometry},
	number = {2},
	pages = {247--270},
	title = {On Open and Closed Convex Codes},
	url = {https://doi.org/10.1007/s00454-018-00050-1},
	volume = {61},
	year = {2019}}

@article{curto2017whatconvex,
author = {Curto, Carina and Gross, Elizabeth and Jeffries, Jack and Morrison, Katherine and Omar, Mohamed and Rosen, Zvi and Shiu, Anne and Youngs, Nora},
title = {What Makes a Neural Code Convex?},
journal = {SIAM Journal on Applied Algebra and Geometry},
volume = {1},
number = {1},
pages = {222-238},
year = {2017},
doi = {10.1137/16M1073170},
URL = {       https://doi.org/10.1137/16M1073170}
}

@article{curto2019algebraic,
	author = {Carina Curto and Elizabeth Gross and Jack Jeffries and Katherine Morrison and Zvi Rosen and Anne Shiu and Nora Youngs},
	doi = {https://doi.org/10.1016/j.jpaa.2018.12.012},
	issn = {0022-4049},
	journal = {Journal of Pure and Applied Algebra},
	number = {9},
	pages = {3919-3940},
	title = {Algebraic signatures of convex and non-convex codes},
	url = {https://www.sciencedirect.com/science/article/pii/S0022404918303025},
	volume = {223},
	year = {2019}}

@article{neuralring13,
	author = {Curto, Carina and Itskov, Vladimir and Veliz-Cuba, Alan and Youngs, Nora},
	doi = {10.1007/s11538-013-9860-3},
	fjournal = {Bulletin of Mathematical Biology},
	issn = {0092-8240},
	journal = {Bull. Math. Biol.},
	number = {9},
	pages = {1571--1611},
	title = {The neural ring: an algebraic tool for analyzing the intrinsic structure of neural codes},
	url = {https://doi.org/10.1007/s11538-013-9860-3},
	volume = {75},
	year = {2013}}

@incollection {curto2019neuralringhomomorphisms,
    AUTHOR = {Curto, Carina and Youngs, Nora},
     TITLE = {Neural ring homomorphisms and maps between neural codes},
 BOOKTITLE = {Topological data analysis---the {A}bel {S}ymposium 2018},
    SERIES = {Abel Symp.},
    VOLUME = {15},
     PAGES = {163--180},
 PUBLISHER = {Springer, Cham},
      YEAR = {[2020] \copyright 2020},
      ISBN = {978-3-030-43407-6; 978-3-030-43408-3},
   MRCLASS = {13P25 (52A37 92B20)},
     DOI = {10.1007/978-3-030-43408-3\_7},
       URL = {https://doi.org/10.1007/978-3-030-43408-3_7},
}

@article {Geller2024,
    AUTHOR = {Geller, Hugh and R. G., Rebecca},
     TITLE = {Canonical forms of neural ideals},
   JOURNAL = {Matematica},
  FJOURNAL = {La Matematica. Official Journal of the Association for Women
              in Mathematics},
    VOLUME = {3},
      YEAR = {2024},
    NUMBER = {2},
     PAGES = {721--752},
      ISSN = {2730-9657},
MRREVIEWER = {Eduardo\ S\'aenz-de-Cabez\'on},
       DOI = {10.1007/s44007-024-00105-1},
       URL = {https://doi.org/10.1007/s44007-024-00105-1},
}

@article{jeffs2019sunflowers,
	author = {R. Amzi Jeffs},
	doi = {https://doi.org/10.1016/j.aam.2019.101935},
	issn = {0196-8858},
	journal = {Advances in Applied Mathematics},
	pages = {101935},
	title = {Sunflowers of convex open sets},
	url = {https://www.sciencedirect.com/science/article/pii/S0196885819301198},
	volume = {111},
	year = {2019}}

@article{morphisms,
author = {Jeffs, R. Amzi},
title = {Morphisms of Neural Codes},
journal = {SIAM Journal on Applied Algebra and Geometry},
volume = {4},
number = {1},
pages = {99-122},
year = {2020},
doi = {10.1137/18M1205509},
URL = { https://doi.org/10.1137/18M1205509},
}

@phdthesis{jeffs2021phd,
    author = {R. Amzi Jeffs},
    title = {Morphisms, Minors, and Minimal Obstructions to Convexity of
Neural Codes},
    school = {University of Washington},
    year = {2021}
}

@misc{jeffs2025coveringrelationsposetcombinatorial,
      title={Covering Relations in the Poset of Combinatorial Neural Codes}, 
      author={R. Amzi Jeffs and Trong-Thuc Trang},
      year={2025},
     eprint={2512.04241},
      archivePrefix={arXiv},
      primaryClass={math.CO},
      note ={available at \url{https://arxiv.org/abs/2512.04241}},
}

@article{luce1952note,
 ISSN = {00029939, 10886826},
 URL = {http://www.jstor.org/stable/2031888},
 author = {R. Duncan Luce},
 journal = {Proceedings of the American Mathematical Society},
 number = {3},
 pages = {382--388},
 publisher = {American Mathematical Society},
 title = {A Note on Boolean Matrix Theory},
 urldate = {2026-01-26},
 volume = {3},
 year = {1952}
}

@article{plemmons1971generalized,
 ISSN = {00361399},
 URL = {http://www.jstor.org/stable/2099964},
 author = {R. J. Plemmons},
 journal = {SIAM Journal on Applied Mathematics},
 number = {3},
 pages = {426--433},
 publisher = {Society for Industrial and Applied Mathematics},
 title = {Generalized Inverses of Boolean Relation Matrices},
 urldate = {2026-01-26},
 volume = {20},
 year = {1971}
}

\appendix
\section{Trunk Lemmas}\label{sec:trunklemmas}
    We state here several technical lemmas concerning trunks of a code.

    \begin{lem}\label{lem:rootsofintersections}
        Let $S,T$ be two trunks of a code $\cC$. Then $\sqrt{S \cap T} \supseteq \sqrt{S} \cup \sqrt{T} \supseteq \sqrt{S} \cap \sqrt{T}$.
    \end{lem}

    \begin{lem}\label{lem:trunkrootcontainment}
        Let $\cC \subseteq 2^{[n]}$.
        Then $\tk_\cC(i) \cap \tk_\cC(j) \neq \tk_\cC(i)$ if and only if $\sqrt[\cC]{ij} \supsetneq \sqrt[\cC]{i}$. (Equivalently, $\tk_\cC(i) \cap \tk_\cC(j) = \tk_\cC(i)$ iff $\sqrt[\cC]{ij} = \sqrt[\cC]{i}$.)
    \end{lem}
    \begin{proof}
        $(\Rightarrow)$ Suppose $\tk_\cC(i) \cap \tk_\cC(j) \neq \tk_\cC(i)$. Since the intersection on the lefthand side is always a subset of the righthand side, the inequality implies it is a strict subset. This implies that the maximal generator of the LHS (which is $\sqrt[\cC]{ij}$) strictly contains the maximal generator of the RHS (which is $\sqrt[\cC]{i}$).
        
        $(\Leftarrow)$ The strict containment $\sqrt[\cC]{ij} \supsetneq \sqrt[\cC]{i}$ implies the strict containment $\tk_\cC(ij) \subsetneq \tk_\cC(i)$, and the LHS is precisely $\tk_\cC(i) \cap \tk_\cC(j)$.
    \end{proof}

    \begin{lem}[The obvious lemmas]\label{lem:misctrunks}
        Let $\cC \subseteq 2^{[n]}$ be a code.
        \begin{enumerate}
            \item\label{insideownroot} For any $\sigma \subseteq [n]$, $\sigma \subseteq \sqrt[\cC]{\sigma}$.
            \item For any nonempty trunk $T \subseteq \cC$, $T = \tk_\cC(\sqrt{T})$.
            \item\label{rootedtrunk} For any $\sigma \in \cC$, $\sigma = \sqrt[\cC]{\sigma}$.
        \end{enumerate}
    \end{lem}
    \begin{proof}
        (1) Since $\sigma \subseteq \tau$ for all $\tau \in \tk_\cC(\sigma)$, it is contained in their intersection.

        (2) Let $\sigma \in \gens(T)$. From part 1, $\sigma \subseteq \sqrt[\cC]{\sigma}$, and therefore $T = \tk_\cC(\sigma) \supseteq \tk_\cC(\sqrt[\cC]{\sigma})$. On the other hand, the statement $T \subseteq \tk_\cC(\sqrt{T})$ is immediate.

        (3) Since $\sigma \in \cC$, we have $\sigma \in \tk_\cC(\sigma)$ and therefore $\sigma \supseteq \sqrt[\cC]{\sigma}$. From part 1, $\sigma \subseteq \sqrt[\cC]{\sigma}$.
    \end{proof}

\section{All isomorphism classes on at most three neurons}\label{sec:isomorphismclasses}
    \Cref{fig:pcode3,fig:pcode3BMF} illustrate the downset in $\pcode$ generated by codes with minimum neuron number 3. Note that this downset includes some codes with minimum neuron number 4, and in fact there are examples of codes of minimum neuron number 3 covering a code of length 4 such that the covering map is a Boolean matrix factorization.

    \begin{figure}[b]
        \centering
        \includegraphics[width=0.75\linewidth]{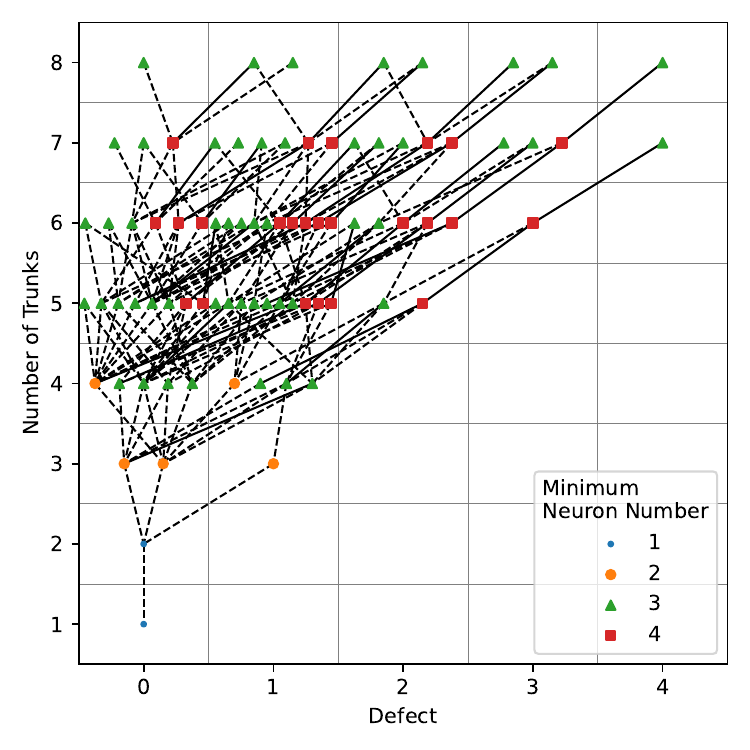}
        \caption{The Hasse diagram of the downset in $\pcode$ generated by all codes of minimum neuron number 3.
        Each gray grid square contains all codes with that combination of trunk number and defect. Solid lines indicate covering relations which are matrix factorizations.}
        \label{fig:pcode3}
    \end{figure}

    \begin{figure}[b]
        \centering
        \includegraphics[width=0.75\linewidth]{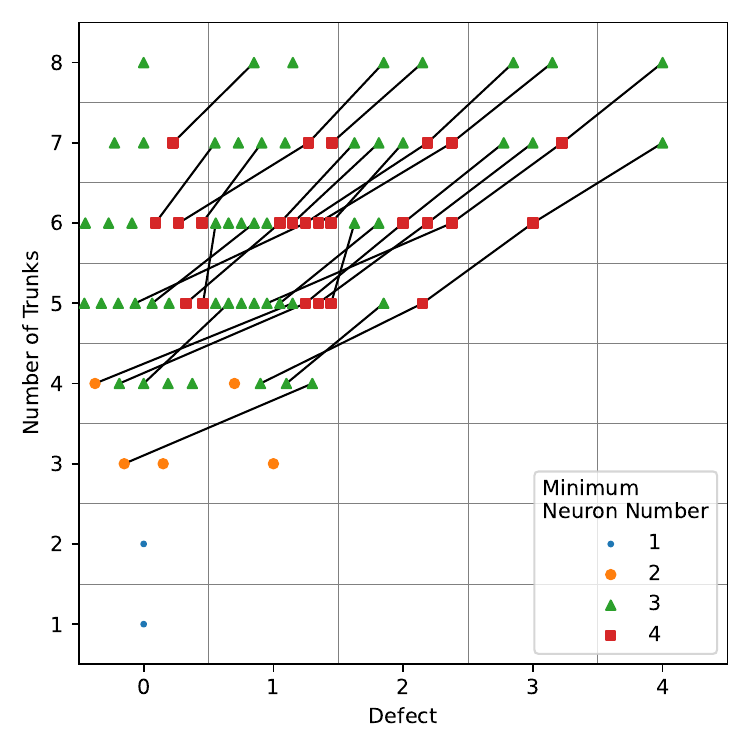}
        \caption{The subset of edges in \cref{fig:pcode3} that correspond to matrix factorizations.}
        \label{fig:pcode3BMF}
    \end{figure}
\end{document}